\documentclass[final]{elsarticle}
\usepackage[top=2.5cm, bottom=2.5cm, left=3cm, right=3cm]{geometry}

\usepackage{hyperref}
\usepackage{graphicx}
\usepackage{amsfonts}
\usepackage{amssymb}
\usepackage{amsmath}
\usepackage{amsthm}
\usepackage{float}
\usepackage{exscale}
\usepackage{relsize}
\usepackage{bm}
\usepackage{subeqnarray}
\usepackage{fancyhdr}
\usepackage{lastpage}
\usepackage{layout}
\usepackage{cleveref}
\usepackage{appendix}
 
\newtheorem{thm}{Theorem}[section]
\newtheorem{corollary}[thm]{Corollary}
\newtheorem{proposition}[thm]{Proposition}
\newtheorem{lemma}[thm]{Lemma}

\newcommand{\bzero}{\mathbf{0}}

\newcommand{\trace}{\mathrm{Tr}}
\newcommand{\MM}{\mathcal{M}}
\newcommand{\diag}{\mathrm{diag}}

\begin{document}

\begin{frontmatter}

\title{\textbf{Generalizing Lieb's Concavity Theorem \\
via Operator Interpolation}}
\author{De Huang\fnref{myfootnote} }
\address{Applied and Computational Mathematics, California Institute of Technology, Pasadena, CA 91125, USA}
\fntext[myfootnote]{E-mail address:\ dhuang@caltech.edu.}
 
\begin{abstract}
We introduce the notion of $k$-trace and use interpolation of operators to prove the joint concavity of the function $(A,B)\mapsto\trace_k\big[(B^\frac{q}{2}K^*A^pKB^\frac{q}{2})^s\big]^\frac{1}{k}$, which generalizes Lieb's concavity theorem and its following results from trace to a class of homogeneous functions $\trace_k[\cdot]^\frac{1}{k}$. Here $\trace_k[A]$ denotes the $k_{\text{th}}$ elementary symmetric polynomial of the eigenvalues of $A$. This result gives an alternative proof for the concavity of $A\mapsto\trace_k\big[\exp(H+\log A)\big]^\frac{1}{k}$ that was obtained and used in a recent work to derive expectation estimates and tail bounds on partial spectral sums of random matrices. 

\end{abstract}


\end{frontmatter}

\section{Introduction}
A fundamental result in the study of trace inequalities is the joint concavity of the function 
\begin{equation}\label{eqt:LCT}
(A,B)\ \longmapsto\ \trace[K^*A^pKB^q]
\end{equation}
on $\mathbf{H}_n^+\times\mathbf{H}_m^+$, for any $K\in \mathbb{C}^{n\times m}$, $p,q\in[0,1], p+q\leq 1$, known as Lieb's Concavity Theorem \cite{LIEB1973267}. Here $\mathbf{H}_n^+$ is the convex cone of all $n \times n$ Hermitian, positive semidefinite matrices. This theorem answered affirmatively an important conjecture by Wigner, Yanase and Dyson \cite{wigner1963information} in information theory. It also led to Lieb's three-matrix extension of the Golden--Thompson inequality, which was then used by Lieb and Ruskai \cite{doi:10.1063/1.1666274} to prove the strong subadditivity of quantum entropy. 

In this paper, we generalize Lieb's concavity theorem from trace to a class of homogeneous matrix functions. In particular, we will prove that the function 
\begin{equation}\label{eqt:GLCT}
(A,B) \ \longmapsto\ \trace_k\big[(B^\frac{q}{2}K^*A^pKB^\frac{q}{2})^s\big]^\frac{1}{k} 
\end{equation}
is jointly concave on $\mathbf{H}_n^+\times\mathbf{H}_m^+$, for any $K\in \mathbb{C}^{n\times m}$ and any $p,q\in[0,1], s\in[0,\frac{1}{p+q}]$. $\trace_k[A]$ denotes the $k_{\text{th}}$ elementary symmetric polynomial of the eigenvalues of $A$, specially $\trace_1[A]=\trace[A]$ and $\trace_n[A]=\det[A]$ if $A\in\mathbb{C}^{n\times n}$. In the case $k=1$, the concavity of function \eqref{eqt:GLCT} has been studied by many and results with an increasing range of $s$ have been obtained over time: $1\leq s\leq \frac{1}{p+q}$ \cite{hiai2001concavity} and $\frac{1}{2}\leq s\leq \frac{1}{p+q}$ \cite{hiai2013concavity} by Hiai, and $0\leq s\leq \frac{1}{1+q}$ by Carlen, Frank and Lieb \cite{carlen2016some}. These partial results together already suffice to conclude the concavity for the full range $0\leq p,q\leq 1,0\leq s\leq \frac{1}{p+q}$. The first complete proof of concavity covering the full range is due to Hiai \cite{hiai2016concavity}. Meanwhile, the convexity of function \eqref{eqt:GLCT} with $k=1$ for different ranges of $p,q,s$ has also been established. A complete convexity/concavity result for the full range of $p,q,s$ was recently accomplished by Zhang \cite{ZHANG2020107053} using an elegant variational approach that is modified from a variational method developed by Carlen and Lieb \cite{Carlen2008}. Here ``full'' means the conditions are also necessary for the corresponding convexity/concavity to hold for all dimensions $n,m$. We recommend the papers \cite{carlen2018inequalities} by Carlen, Frank and Lieb, and \cite{ZHANG2020107053} by Zhang for historical overviews on this topic in more detail. In this paper, we will only work on the concavity of function \eqref{eqt:GLCT}, since the map $A\mapsto\trace_k[A]^\frac{1}{k}$ is homogeneous of order 1 and concave on $\mathbf{H}_n^+$ for $k\geq 2$. Our work therefore extends the established concavity results from the normal trace to a class of $k$-trace functions. 

The motivation of deriving our generalized Lieb's concavity theorem is to provide an alternative proof for the concavity of the map
\begin{equation}\label{eqt:GLT}
A\ \longmapsto\ \trace_k\big[\exp(H+\log A)\big]^\frac{1}{k}
\end{equation}
on $\mathbf{H}_n^{++}$ (positive definite), for any Hermitian matrix $H$ of the same size, due to a recent work by Huang \cite{HUANG2019419}, and hence completing the theory behind it. For $k=1$, the concavity of 
\begin{equation}\label{eqt:LT}
A\ \longmapsto\ \trace\big[\exp(H+\log A)\big],
\end{equation}
also due to Lieb \cite{LIEB1973267}, is an equivalence of Lieb's concavity theorem. Tropp \cite{Tropp2012,MAL-048} made use of the concavity of \eqref{eqt:LT} to establish his master bounds on the largest (or smallest) eigenvalue of a sum of random matrices. Huang \cite{HUANG2019419} introduced the notion of $k$-trace functions $\trace_k$ to provide estimates on partial spectral sums of Hermitian matrices, and then used the concavity of \eqref{eqt:GLT} to generalize Tropp's master bounds from the largest (or smallest) eigenvalue to the sum of the $k$ largest (or smallest) eigenvalues. Huang's proof of the concavity of \eqref{eqt:GLT} in \cite{HUANG2019419} was an imitation of Lieb's original arguments using matrix derivatives, which, however, failed to extend to the more complicated function \eqref{eqt:GLCT}. We then looked for more profound approaches to see a bigger picture.

Since Lieb's original establishment of his concavity theorem, alternative proofs have been developed from different aspects of matrix theories, including matrix tensors (Ando \cite{ANDO1979203}, Carlen \cite{carlen2010trace}, Nikoufar et al. \cite{NIKOUFAR2013531}), the theory of Herglotz functions (Epstein \cite{epstein1973remarks}), and interpolation theories (Uhlmann \cite{uhlmann1977relative}, Kosaki \cite{kosaki1982interpolation}). The tensor approaches prove the theorem elegantly by translating the concavity of \eqref{eqt:LCT} to the operator concavity of the map $(A,B)\mapsto A^p\otimes B^q$, but have difficulties in generalizing to our $k$-trace case due to the nonlinearity of $\trace_k$. However, the $k$-trace has two good properties that are most essential to the desired results: (i) the map $A\mapsto\trace_k[A]^\frac{1}{k}$ is concave on $\mathbf{H}_n^+$, and (ii) the $k$-trace satisfies H\"older's inequality as the normal trace does. We, therefore, turned to the more generalizable methods of operator interpolation based essentially on H\"older's inequality. Originating from the Hadamard three-lines theorem \cite{hadamard1899theoreme}, the interpolation of operators has been a powerful tool in operator and functional analysis, with variant versions including the Riesz--Thorin interpolation theorem \cite{riesz1926maxima}, Stein's interpolation of holomorphic operators \cite{stein1956interpolation}, Peetre's K-method \cite{peetre1963theory} and many others. In particular, we found Stein's complex interpolation technique most compatible and easiest to use in the $k$-trace setting. Our use of interpolation technique was inspired by a recent work of Sutter et al. \cite{sutter2017multivariate}, in which they applied Stein's interpolation to derive a multivariate extension of the Golden--Thompson inequality. This interpolation technique will help us first prove a key lemma that the function 
\begin{equation}\label{eqt:GE}
A \ \longmapsto\ \trace_k\big[(K^*A^pK)^s\big]^\frac{1}{k}
\end{equation}
is concave on $\mathbf{H}_n^+$, for any $K\in \mathbb{C}^{n\times n}$ and any $p\in[0,1],s\in[0,\frac{1}{p}]$. Note that function \eqref{eqt:GE} is a special case of function \eqref{eqt:GLCT} with $q=0$. Given this lemma, the concavity in the more general case can be obtained via a powerful variational argument that originates in \cite{Carlen2008} by Carlen and Lieb. This kind of variational methods, introducing supremum/infimum characterizations of trace functions, has been widely used in the study of the convexity/concavity of function \eqref{eqt:GLCT} and its variants in the trace case (see e.g. \cite{carlen2010trace,Carlen2008,carlen2016some,carlen2018inequalities}). Our proof for the $k$-trace case will be imitating a refined variational approach by Zhang \cite{ZHANG2020107053}, which is again based on H\"older's inequalities.

\subsection*{outline}
The rest of the paper is organized as follows. \Cref{sec:Notations} is devoted to introductions of general notations and the notion of $k$-trace. We will present our main theorems and their proofs in \Cref{sec:Main}, and briefly review previous proofs of Lieb's original concavity theorem. We also introduce in \Cref{sec:Main} the operator interpolation technique that we will use. Two other results on $k$-trace will be given in \Cref{sec:OtherResults}. Some details of background knowledge are discussed in the appendices. 

\section{Notations and $k$-trace}
\label{sec:Notations}

\subsection{General conventions}
For any positive integers $n,m$, we write $\mathbb{C}^n$ for the $n$-dimensional complex vector spaces equipped with the standard $l_2$ inner products, and $\mathbb{C}^{n\times m}$ for the space of all complex matrices of size $n\times m$. Let $\mathbf{H}_n$ be the space of all $n\times n$ Hermitian matrices, $\mathbf{H}_n^+$ be the convex cone of all $n\times n$ Hermitian, positive-semidefinite matrices, and $\mathbf{H}_n^{++}$ be the convex cone of all $n\times n$ Hermitian, positive definite matrices. We denote the Loewner partial orders on $\mathbf{H}_n$ in the usual way: for any $A,B\in \mathbf{H}_n$, we write $A\succeq B$ or $B\preceq A$ if $A-B\in \mathbf{H}_n^+$, and write $A\succ B$ or $B\prec A$ if $A-B\in \mathbf{H}_n^{++}$. We write $\bf{0}$ for square zero matrices of suitable size according to the context, and $I_n$ for the identity matrix of size $n \times n$.

Following the notations in \cite{HUANG2019419}, the $k$-trace of a matrix $A\in\mathbb{C}^{n\times n}$ is defined as 
\begin{equation}
\label{def:k-trace}
\trace_k[A] = \sum_{1\leq i_1<i_2<\cdots<i_k\leq n}\lambda_{i_1}\lambda_{i_2}\cdots\lambda_{i_k},\qquad 1\leq k\leq n,
\end{equation}
where $\lambda_1,\lambda_2,\cdots,\lambda_n$ are all eigenvalues of $A$. In particular, $\trace_1[A]=\trace[A]$ is the normal trace of $A$, and $\trace_n[A]=\det[A]$ is the determinant of $A$. If we write $A_{(i_1\cdots i_k,i_1\cdots i_k)}$ for the $k\times k$ principal submatrix of $A$ corresponding to the indices $i_1,i_2,\cdots,i_k$, then an equivalent definition of the $k$-trace of $A$ is given by 
\begin{equation}
\label{def:k-trace2}
\trace_k[A] = \sum_{1\leq i_1<i_2<\cdots<i_k\leq n}\det[A_{(i_1\cdots i_k,i_1\cdots i_k)}],\qquad 1\leq k\leq n.
\end{equation}
Using the second definition \eqref{def:k-trace2}, one can check that for any $1\leq k\leq n$, the $k$-trace enjoys the cyclic invariance like the normal trace and the determinant. That is for any $A,B\in \mathbb{C}^{n\times n}$, $\trace_k[AB]=\trace_k[BA]$.

For any function $f:\mathbb{R}\rightarrow\mathbb{R}$, the extension of $f$ to a function from $\mathbf{H}_n$ to $\mathbf{H}_n$ is given by 
\[f(A)=\sum_{i=1}^nf(\lambda_i)u_iu_i^*, \quad A\in \mathbf{H}_n,\]
where $\lambda_1,\lambda_2,\cdots,\lambda_n$ are the eigenvalues of $A$, and $u_1,u_2,\cdots,u_n\in\mathbb{C}^n$ are the corresponding normalized eigenvectors. A function $f$ is said to be operator monotone increasing (or decreasing) if $A\succeq B$ implies $f(A)\succeq f(B)$ (or $f(A)\preceq f(B)$); $f$ is said to be operator convex (or concave) on some convex set $S$, if
\[\tau f(A)+(1-\tau) f(B) \succeq f(\tau A+(1-\tau) B)\ (\text{or} \preceq f(\tau A+(1-\tau) B)),\] 
for any $A,B\in S$ and any $\tau\in[0,1]$. For example, the function $A\mapsto A^r$ is both operator monotone increasing and operator concave on $\mathbf{H}_n^+$ for $r\in[0,1]$ (the Loewner--Heinz theorem \cite{lowner1934monotone}, \cite{heinz1951beitrage}, \cite{kraus1936konvexe}, see also \cite{carlen2010trace}). One can find more details and properties of matrix functions in \cite{carlen2010trace,vershynina2013matrix}. For any $A\in\mathbb{C}^{n\times m}$, we denote by $\|A\|_p$ the standard Schatten $p$-norm,
\begin{equation}
\|A\|_p = \trace[|A|^p]^\frac{1}{p},
\end{equation}
where $|A|=(A^*A)^\frac{1}{2}$. In particular, we write $\|A\| = \|A\|_{\infty}=$ the largest singular value of $A$.

\subsection{$k$-trace}
Huang \cite{HUANG2019419} introduced the notion of $k$-trace to provide bounds on the sum of the $k$ largest (or smallest) eigenvalues of a matrix $A\in \mathbf{H}_n$, 
\begin{align*}
&\sum_{i=1}^k\lambda_i(A) \leq \log \trace_k \big[\exp(A)\big]\leq \sum_{i=1}^k\lambda_i(A) + \log \binom{n}{k},\\
&\sum_{i=1}^k\lambda_{n-i+1}(A) \geq -\log \trace_k \big[\exp(-A)\big]\geq \sum_{i=1}^k\lambda_{n-i+1}(A) - \log \binom{n}{k},
\end{align*}
where $\lambda_i(A)$ denotes the $i_{\text{th}}$ largest eigenvalue of $A$. Huang then used these estimates and the concavity of $A\mapsto \trace_k[\exp(H+\log A)]^\frac{1}{k}$ to derive expectation estimates and tail bounds for the sum of the $k$ largest (or smallest) eigenvalues of a class of random matrices. Apart from this particular application, the $k$-trace is of theoretical interest in itself, as it has many interpretations corresponding to different aspects of matrix theories. Writing $D(A^{(1)},A^{(2)},\cdots,A^{(n)})$ the mixed discriminant of any $n$ matrices $A^{(1)},A^{(2)},\cdots,A^{(n)}\in \mathbb{C}^{n\times n}$, we then have the identity 
\[\trace_k[A] = \binom{n}{k}D(\underbrace{A,\cdots,A}_{k},\underbrace{I_n,\cdots,I_n}_{n-k}),\]
which as well connects the $k$-trace to the notion of the $k_\text{th}$ intrinsic volume in convex geometry. Also, if we consider the $k_\text{th}$ exterior algebra $\wedge^k(\mathbb{C}^n)$, we can then interpret the $k$-trace of $A$ as 
\[\trace_k[A] = \trace_{\mathcal{L}(\wedge^k(\mathbb{C}^n))}\big[\MM^{(k)}_0(A)\big],\]
where $\trace_{\mathcal{L}(\wedge^k(\mathbb{C}^n))}$ is the normal trace on the operator space $\mathcal{L}(\wedge^k(\mathbb{C}^n))$, and $\MM^{(k)}_0(A)\in \mathcal{L}(\wedge^k(\mathbb{C}^n))$ is defined as $\MM^{(k)}_0(A)(v_1\wedge v_2\wedge\cdots \wedge v_k) = Av_1\wedge Av_2\wedge\cdots \wedge Av_k$, for any $v_1\wedge v_2\wedge\cdots \wedge v_k\in \wedge^k(\mathbb{C}^n)$. More discussions on these two interpretations and how they can be used to study the $k$-trace will be presented in \Cref{Apdix:MixedDiscriminant} and \Cref{Apdix:ExteriorAlgebra}. 

Throughout the paper, we will be using the following properties of the $k$-trace.
\begin{proposition}
\label{prop:ktrace}
For any positive integers $n,k$, $1\leq k\leq n$, the $k$-trace function $\trace_k[\cdot]$ satisfies the following:
\begin{itemize}
\item[(i)] Cyclicity: $\trace_k[AB] = \trace_k[BA]$, $A,B\in \mathbb{C}^{n\times n}$.
\item[(ii)] Homogeneity: $\trace_k[\alpha A] = \alpha^k\trace_k[A]$, $A\in \mathbb{C}^{n\times n}, \alpha\in \mathbb{C}$.
\item[(iii)] Monotonicity: For any $A,B\in \mathbf{H}_n^+$, $\trace_k[A]\geq \trace_k[B]$, if $A\succeq B$; $\trace_k[A]> \trace_k[B]$, if $A\succ B$. In particular, $\trace_k[A]\geq 0, A\in \mathbf{H}_n^+$.  
\item[(iv)] Concavity: The function $A\mapsto (\trace_k[A])^\frac{1}{k}$ is concave on $\mathbf{H}_n^+$.
\item[(v)] H\"older's Inequality: $\trace_k[|AB|^r]^\frac{1}{r}\leq \trace_k[|A|^p]^\frac{1}{p}\trace_k[|B|^q]^\frac{1}{q}$, for any $r,p,q\in(0,+\infty],\frac{1}{p}+\frac{1}{q}=\frac{1}{r}$, and any $A,B\in \mathbb{C}^{n \times n}$.
\item[(vi)] Consistency: For any $\tilde{n}$, $k\leq \tilde{n}\leq n$, and any $A\in \mathbb{C}^{\tilde{n}\times \tilde{n}}$, $\trace_k\left[\left(\begin{array}{cc}
A & \bm{0} \\
\bm{0} & \bm{0}
\end{array}\right)_{n\times n}\right]
=\trace_k[A]$. 
\end{itemize}
\end{proposition}

\begin{proof}
(i), (ii), (iii) and (vi) can be easily verified by the definitions \eqref{def:k-trace} and \eqref{def:k-trace2}. (iv) is a consequence of the general Brunn--Minkowski theorem (\Cref{cor:BMtheorem}) in \Cref{Apdix:MixedDiscriminant}. (v) is a direct result of expression \eqref{eqt:extrace2ktrace} in \Cref{Apdix:ExteriorAlgebra}. In fact, since the normal trace enjoys the H\"older's inequality, we have
\begin{align*}
\trace_k[|AB|^r]^\frac{1}{r}=&\ \trace[\big|\MM_0^{(k)}(A)\MM_0^{(k)}(B)\big|^r]^\frac{1}{r} \\
\leq&\ \trace[|\MM_0^{(k)}(A)|^p]^\frac{1}{p}\trace[|\MM_0^{(k)}(B)|^q]^\frac{1}{q}\\
=&\ \trace_k[|A|^p]^\frac{1}{p}\trace_k[|B|^q]^\frac{1}{q}.
\end{align*}
We have used multiple properties of the operator $\MM_0^{(k)}(A)$ introduced in \Cref{Apdix:ExteriorAlgebra}.
\end{proof}

\section{Generalizing Lieb's Concavity Theorem}
\label{sec:Main}

\subsection{Main Theorems}
In what follows, we will fix the integer $k$ and always write 
\[\phi(A) = (\trace_k[A])^\frac{1}{k}\]
for simplicity. Note that the function $\phi$ also satisfies (i) cyclicity, (iii) monotonicity, (v) H\"older's inequality and (vi) consistency as in \Cref{prop:ktrace}. But now the map $A\mapsto \phi(A)$ is homogeneous of order 1 and is concave on $\mathbf{H}_n^+$. Abusing notation, we will also refer the function $\phi$ as the $k$-trace. Our main results of this paper are the following.

\begin{lemma}\label{lem:GeneralEpstein}
For any $r\in[0,1],s\in[0,\frac{1}{r}]$ and any $K\in \mathbb{C}^{n\times n}$, the function
\begin{equation}
A \ \longmapsto\ \phi\big((K^*A^rK)^s\big) 
\label{eqt:function1}
\end{equation}
is concave on $\mathbf{H}_n^+$. 
\end{lemma}

\begin{thm}[Generalized Lieb's Concavity Theorem]\label{thm:GeneralLiebConcavity}
For any $p,q\in[0,1],s\in[0,\frac{1}{p+q}]$, and any $K\in \mathbb{C}^{n\times m}$, the function 
\begin{equation}
(A,B) \ \longmapsto\ \phi\big((B^\frac{q}{2}K^*A^pKB^\frac{q}{2})^s\big) 
\label{eqt:function2}
\end{equation}
is jointly concave on $\mathbf{H}_n^+\times\mathbf{H}_m^+$. 
\end{thm}

\begin{thm}\label{thm:GeneralLieb}
For any $H\in \mathbf{H}_n$ and any $\{p_j\}_{j=1}^m\subset[0,1]$ such that $\sum_{j=1}^mp_j\leq1$, the function 
\begin{equation}
(A^{(1)},A^{(2)},\dots,A^{(m)}) \ \longmapsto\ \phi\big(\exp\big(H+\sum_{j=1}^mp_j\log A^{(j)}\big)\big)
\label{eqt:function3}
\end{equation}
is jointly concave on $(\mathbf{H}_n^{++})^{\times m}$. In particular, $A\mapsto\phi\big(\exp(H+\log A)\big)$ is concave on $\mathbf{H}_n^{++}$.
\end{thm}

\Cref{lem:GeneralEpstein} is a $k$-trace extension of the concave part of Lemma 2.8 in \cite{Carlen2008} (see also \cite[Theorem 4.1]{hiai2013concavity}). The latter is a consequence of Lieb's original concavity theorem. However, we will first apply the technique of operator interpolation to prove \Cref{lem:GeneralEpstein} independently, and then use it to derive \Cref{thm:GeneralLiebConcavity} and the other results. In fact, the convexity/concavity of function \eqref{eqt:function1} in the trace case with different ranges of $p,s$ has been used as the first step towards many consequential results on more complicated trace functions. \Cref{thm:GeneralLiebConcavity} is our generalized Lieb's concavity theorem, which extends Hiai's Theorem 2.1 in \cite{hiai2016concavity} (see also \cite[Theorem 4.4]{carlen2016some}) from trace to $k$-trace. Note that, as stated in Hiai's Theorem, the concavity also holds for $-1\leq p,q\leq 0,\frac{1}{p+q}\leq s\leq0$. In fact,  by consistency and continuity of $\phi$, it suffices to consider the case when $n=m$ and $A,B,K$ are invertible. Then, following the discussions in \cite{carlen2018inequalities}, we have that 
\[\phi\big((B^\frac{q}{2}K^*A^pKB^\frac{q}{2})^s\big) = \phi\big((B^{-\frac{q}{2}}K^{-1}A^{-p}(K^{-1})^*B^{-\frac{q}{2}})^{-s}\big),\]
which concludes the concavity for the mirrored range of $p,q,s$. Our derivation from \Cref{lem:GeneralEpstein} to \Cref{thm:GeneralLiebConcavity} will be a counterpart of Zhang's simple and useful variational argument in \cite{ZHANG2020107053} for the trace case. 

\Cref{thm:GeneralLieb} is a generalization of both Corollary 6.1 in \cite{LIEB1973267} (from trace to $k$-trace) and Theorem 2.1 in \cite{HUANG2019419} (from univariate to multivariate). Lieb \cite{LIEB1973267} proved the original trace version by checking the non-positiveness of the second order directional derivatives (or Hessians). Huang \cite{HUANG2019419} imitated Lieb's derivative arguments and proved the concavity of $A\mapsto\phi\big(\exp(H+\log A)\big)$, which he then used to derive concentrations of partial spectral sums of random matrices. We will first prove \Cref{thm:GeneralLieb} for $m=1$ by applying the Lie product formula to \Cref{lem:GeneralEpstein} (taking $p,q\rightarrow 0,s\rightarrow +\infty$), and hence providing an alternative proof of the concavity of $A\mapsto\phi\big(\exp(H+\log A)\big)$. We then improve the result from $m=1$ to $m\geq 1$ using a $k$-trace version of the Araki--Lieb--Thirring inequality (\Cref{lem:ktraceALT}). All proofs of our main results will be presented in \Cref{subsec:Proofs}.

Apart from the above theorems, we will also prove (i) a $k$-trace version of the multivariate extension of the Golden--Thompson inequality, and (ii) the monotonicity preserving and concavity preserving properties of $k$-trace. \Cref{sec:OtherResults} will be devoted to these two results.  

\subsection{Operator Interpolation}

One of our main tools is Stein's interpolation of linear operators \cite{stein1956interpolation}, that was developed from Hirschman's stronger version of the Hadamard three-line theorem \cite{hirschman1952convexity}. This technique was recently adopted by Sutter et al. \cite{sutter2017multivariate} to establish a multivariate extension of the Golden--Thompson inequality, which inspired our use of interpolation in proving the generalized Lieb's concavity theorem. We will follow the notations in \cite{sutter2017multivariate}. For any $\theta\in(0,1)$, we define a density $\beta_\theta(t)$ on $\mathbb{R}$ by 
\begin{equation} 
\beta_\theta(t) = \frac{\sin(\pi\theta)}{2\theta\big(\cosh(\pi t)+\cos(\pi\theta)\big)},\quad t\in\mathbb{R}.
\end{equation} 
Specially, we define 
\[\beta_0(t) = \lim_{\theta\searrow0}\beta_\theta(t) = \frac{\pi}{2(\cosh(\pi t)+1)},\quad \text{and}\quad \beta_1(t) = \lim_{\theta\nearrow1}\beta_\theta(t) = \delta(t).\]
$\beta_\theta(t)$ is a density since $\beta_\theta(t)\geq0,t\in\mathbb{R}$ and $\int_{-\infty}^{+\infty}\beta_\theta(t)dt=1$. We will always use $\mathcal{S}$ to denote a vertical strip on the complex plane $\mathbb{C}$:
\begin{equation}\label{eqt:S}
\mathcal{S}=\{z\in \mathbb{C}:0\leq \mathrm{Re}(z)\leq 1\}.
\end{equation}

\begin{thm}[Stein-Hirschman]\label{thm:SHInterpolation}
Let $G(z)$ be a map from $\mathcal{S}$ to bounded linear operators on a separable Hilbert space that is holomorphic in the interior of $\mathcal{S}$ and continuous on the boundary. Let $p_0,p_1\in[1,+\infty],\theta\in[0,1]$, and define $p_\theta$ by 
\[\frac{1}{p_\theta}=\frac{1-\theta}{p_0}+\frac{\theta}{p_1}.\]
Then if $\|G(z)\|_{p_{\mathrm{Re}(z)}}$ is uniformly bounded on $\mathcal{S}$, the following inequality holds:
\begin{equation}\label{eqt:SHInterpolation}
\log\|G(\theta)\|_{p_\theta} \leq  \int_{-\infty}^{+\infty}dt\Big(\beta_{1-\theta}(t)\log\|G(it)\|_{p_0}^{1-\theta}+\beta_\theta(t)\log\|G(1+it)\|_{p_1}^\theta\Big).
\end{equation}
\end{thm}

A $k$-trace analog of the above theorem, that is more convenient for our use, is as follows, 

\begin{lemma}\label{lem:KeyLemma}
Let $G(z):\mathcal{S}\rightarrow\mathbb{C}^{n\times n}$ be holomorphic in the interior of $\mathcal{S}$ and continuous on the boundary. Let $p_0,p_1\in [1,+\infty]$, $\theta\in[0,1]$, and define $p_\theta$ by 
\[\frac{1}{p_\theta}=\frac{1-\theta}{p_0}+\frac{\theta}{p_1}.\]
Then if $\|G(z)\|$ is uniformly bounded on $\mathcal{S}$, the following inequality holds:
\begin{align}\label{eqt:KeyIneq}
&\ \log\big[\phi(|G(\theta)|^{p_\theta})^\frac{1}{p_\theta}\big] \nonumber \\
\leq&\  \int_{-\infty}^{+\infty}dt\Big(\beta_{1-\theta}(t)\log\big[\phi\big(|G(it)|^{p_0}\big)^\frac{1-\theta}{p_0}\big]+\beta_\theta(t)\log\big[\phi\big(|G(1+it)|^{p_1}\big)^\frac{\theta}{p_1}\big]\Big).
\end{align}
\end{lemma}

More discussions on operator interpolation and the proof of \Cref{lem:KeyLemma} will be presented in \Cref{Apdix:ComplexInterpolation}. For $p_0,p_1\in[1,+\infty)$, we can rewrite inequality \eqref{eqt:KeyIneq} as 
\begin{align}\label{eqt:KeyIneq2}
&\ \log\phi(|G(\theta)|^{p_\theta}) \nonumber\\ 
\leq&\   \int_{-\infty}^{+\infty}dt\Big(\frac{(1-\theta)p_\theta}{p_0}\beta_{1-\theta}(t)\log\phi\big(|G(it)|^{p_0}\big)+\frac{\theta p_\theta}{p_1}\beta_\theta(t)\log\phi\big(|G(1+it)|^{p_1}\big)\Big)
\end{align}
Notice that 
\[\int_{-\infty}^{+\infty}\Big(\frac{(1-\theta)p_\theta}{p_0}\beta_{1-\theta}(t)+\frac{\theta p_\theta}{p_1}\beta_\theta(t)\Big)dt = 1.\]
Then using Jensen's inequality on the concavity of logarithm, we can immediately conclude from \eqref{eqt:KeyIneq2} that for $p_0,p_1\in[1,+\infty)$,
\begin{equation}\label{eqt:KeyIneq3}
\phi(|G(\theta)|^{p_\theta})\leq \int_{-\infty}^{+\infty}dt\Big(\frac{(1-\theta)p_\theta}{p_0}\beta_{1-\theta}(t)\phi\big(|G(it)|^{p_0}\big)+\frac{\theta p_\theta}{p_1}\beta_\theta(t)\phi\big(|G(1+it)|^{p_1}\big)\Big),
\end{equation}
under the same setting as in \Cref{lem:KeyLemma}.

\subsection{Proof of Main Results}
\label{subsec:Proofs}

The key of applying \Cref{lem:KeyLemma} is to choose some proper holomorphic function $G(z)$ and then interpolating on some power in $[0,1]$. In particular, we will perform interpolation on $w=\frac{1}{s}$ to prove \Cref{lem:GeneralEpstein}. Our choice of the holomorphic functions $G(z)$ in the following proof is inspired by Lieb's constructions in \cite{LIEB1973267} for the use of maximum modulus principle.

\begin{proof}[\rm\textbf{Proof of \Cref{lem:GeneralEpstein}}] Note that for $s\in [0,1]$, the concavity of \eqref{eqt:function1} is a direct consequence of the facts that (i) $\phi$ is monotone increasing and concave on $\mathbf{H}_n^{+}$, and (ii) $X\mapsto X^r$ and $X\mapsto X^s$ are operator monotone increasing and operator concave on $\mathbf{H}_n^{+}$. So in what follows we may assume that $1\leq s\leq \frac{1}{r}$. We need to show that, for any $A,B\in \mathbf{H}_n^+$ and any $\tau\in[0,1]$, 
\[\tau\phi\big((K^*A^rK)^s\big)  +(1-\tau)\phi\big((K^*B^rK)^s\big)  \leq \phi\big((K^*C^rK)^s\big)  ,\]
where $C=\tau A+(1-\tau)B$. We may assume that $A,B\in \mathbf{H}_n^{++}$ and $K$ is invertible. Once this is done, the general result for $A,B\in \mathbf{H}_n^+$ and $K\in\mathbb{C}^{n\times n}$ can be obtained by continuity. Let $w=\frac{1}{s}\in[r,1]$ and $\hat{r} = rs\in[0,1]$, so $r = \hat{r}w$. Let $M = C^\frac{r}{2}K$, and let $M=Q|M|$ be the polar decomposition of $M$ for some unitary matrix $Q$. Since $C,K$ are both invertible, $|M|\in\mathbf{H}_n^{++}$. We then define two functions from $\mathcal{S}$ to $\mathbb{C}^{n\times n}$:
\[G_A(z) = A^\frac{\hat{r}z}{2}C^{-\frac{\hat{r}z}{2}}Q|M|^\frac{z}{w},\quad G_B(z) = B^\frac{\hat{r}z}{2}C^{-\frac{\hat{r}z}{2}}Q|M|^\frac{z}{w},\quad z\in\mathcal{S},\]
where $\mathcal{S}$ is given by \eqref{eqt:S}. In what follows we will use $X$ for $A$ or $B$. We then have
\begin{align*}
\phi\big((K^*X^rK)^s\big) =&\  \phi\big((M^*C^{-\frac{r}{2}}X^rC^{-\frac{r}{2}}M)^s\big)\\
=&\ \phi\big((|M|Q^*C^{-\frac{\hat{r}w}{2}}X^\frac{\hat{r}w}{2}X^\frac{\hat{r}w}{2}C^{-\frac{\hat{r}w}{2}}Q|M|)^\frac{1}{w}\big)\\
=&\ \phi\big(|G_X(w)|^\frac{2}{w}\big).
\end{align*}
Since $A,B,C,M$ are now fixed matrices in $\mathbf{H}_n^{++}$, $G_A(z)$ and $G_B(z)$ are apparently holomorphic in the interior of $\mathcal{S}$ and continuous on the boundary. Also, it is easy to check that $\|G_A(z)\|$ and $\|G_B(z)\|$ are uniformly bounded on $\mathcal{S}$, since $\mathrm{Re}(z)\in[0,1]$. Therefore we can use inequality \eqref{eqt:KeyIneq3} with $\theta = w,p_\theta = \frac{2}{w}$ to obtain 
\begin{align*}
&\ \phi(|G_X(w)|^\frac{2}{w})\leq \int_{-\infty}^{+\infty}dt\Big(\frac{2(1-w)}{wp_0}\beta_{1-w}(t)\phi\big(|G_X(it)|^{p_0}\big)+\frac{2}{p_1}\beta_w(t)\phi\big(|G_X(1+it)|^{p_1}\big)\Big).
\end{align*}
We still need to choose some $p_0,p_1\geq 1$ satisfying $\frac{1-w}{p_0}+\frac{w}{p_1}=\frac{1}{p_w}=\frac{w}{2}$ to proceed. Note that $G_X(it) = X^\frac{i\hat{r}t}{2}C^{-\frac{i\hat{r}t}{2}}Q|M|^\frac{it}{w}$ are now unitary matrices for all $t\in\mathbb{R}$ since $X,C,|M|\in\mathbf{H}_n^{++}$, and thus $|G_X(it)|^{p_0} = I_n$ for all $p_0$. Therefore we can take $p_0\rightarrow +\infty,p_1=2$ to obtain 
\[\phi(|G_X(w)|^\frac{2}{w})\leq \int_{-\infty}^{+\infty}dt\beta_w(t)\phi\big(|G_X(1+it)|^2\big).\]
Further, for each $t\in\mathbb{R}$, we have 
\begin{align*}
&\ \phi\big(|G_X(1+it)|^2\big)\\
=&\ \phi\big(G_X(1+it)^*G_X(1+it)\big) \\
= &\ \phi\big(|M|^\frac{(1-it)}{w}Q^*C^{-\frac{\hat{r}(1-it)}{2}}X^{\hat{r}}C^{-\frac{\hat{r}(1+it)}{2}}Q|M|^\frac{(1+it)}{w}\big)\\
=&\ \phi\big(|M|^\frac{1}{w}Q^*C^{-\frac{\hat{r}(1-it)}{2}}X^{\hat{r}}C^{-\frac{\hat{r}(1+it)}{2}}Q|M|^\frac{1}{w}\big),
\end{align*}
where we have used the cyclicity of $\phi$ ($|M|^\frac{it}{w}$ is unitary) for the last equality. Therefore we have
\begin{align*}
&\ \tau\phi\big(|G_A(1+it)|^2\big) + (1-\tau)\phi\big(|G_B(1+it)|^2\big) \\
=&\ \tau\phi\big(|M|^\frac{1}{w}Q^*C^{-\frac{\hat{r}(1-it)}{2}}A^{\hat{r}}C^{-\frac{\hat{r}(1+it)}{2}}Q|M|^\frac{1}{w}\big)\\
&\ + (1-\tau)\phi\big(|M|^\frac{1}{w}Q^*C^{-\frac{\hat{r}(1-it)}{2}}B^{\hat{r}}C^{-\frac{\hat{r}(1+it)}{2}}Q|M|^\frac{1}{w}\big)\\
\leq &\ \phi\big(|M|^\frac{1}{w}Q^*C^{-\frac{\hat{r}(1-it)}{2}}(\tau A^{\hat{r}}+(1-\tau) B^{\hat{r}})C^{-\frac{\hat{r}(1+it)}{2}}Q|M|^\frac{1}{w}\big)\\
\leq&\ \phi\big(|M|^\frac{1}{w}Q^*C^{-\frac{\hat{r}(1-it)}{2}}C^{\hat{r}}C^{-\frac{\hat{r}(1+it)}{2}}Q|M|^\frac{1}{w}\big)\\
=&\ \phi\big(|M|^\frac{2}{w}\big)\\
=&\ \phi\big((M^*M)^\frac{1}{w}\big).
\end{align*}
The first inequality above is due to the concavity of $\phi$, the second inequality is due to (i) that $\phi$ is monotone increasing on $\mathbf{H}_n^+$ and (ii) that $X\mapsto X^{\hat{r}}$ is operator concave on $\mathbf{H}_n^+$ for $\hat{r}\in(0,1]$. Finally, since $\phi\big((M^*M)^\frac{1}{w}\big)$ is independent of $t$, and $\beta_w(t)$ is a density on $\mathbb{R}$, we obtain that 
\begin{align*}
&\ \tau\phi\big((K^*A^rK)^s\big)  +(1-\tau)\phi\big((K^*B^rK)^s\big) \\
=&\ \tau\phi\big(|G_A(w)|^\frac{2}{w}\big) + (1-\tau)\phi\big(|G_B(w)|^\frac{2}{w}\big) \\ 
\leq&\ \phi\big((M^*M)^\frac{1}{w}\big)\\
=&\ \phi\big((K^*C^rK)^s\big).
\end{align*}
So we have proved the concavity of \eqref{eqt:function1} on $\mathbf{H}_n^+$.
\end{proof}

Our next proof, using essentially H\"older's inequalities for the $k$-trace, is adapted from Zhang's proofs of Theorem 1.1 and Theorem 3.3 in \cite{ZHANG2020107053}.

\begin{proof}[\rm\textbf{Proof of \Cref{thm:GeneralLiebConcavity}}]
Without loss of generality, we may assume that $m=n$. Otherwise we can replace $A$ by $\left(\begin{array}{cc} A & \bzero \\ \bzero & \bzero \end{array}\right)$ and $K$ by $\left(\begin{array}{c} K \\ \bzero\end{array}\right)$ if $n<m$; or replace $B$ by $\left(\begin{array}{cc} B & \bzero \\ \bzero & \bzero \end{array}\right)$ and $K$ by $\left(\begin{array}{cc} K & \bzero\end{array}\right)$ is $n>m$. By the consistency of $\phi$, these changes of variables will not affect whether the function \eqref{eqt:function2} is jointly concave in $(A,B)$ or not. We write $X=A^\frac{p}{2}$ and $Y = KB^\frac{q}{2}$. Let $s_1 = \frac{p+q}{p}s,s_2 = \frac{p+q}{q}s$, so $\frac{1}{s} = \frac{1}{s_1}+\frac{1}{s_2}$. Then for any $Z\in \mathbb{C}^{n\times n}$ that is invertible, we have by H\"older's inequality ((v) in \Cref{prop:ktrace}) that 
\begin{align*}
\phi\big((B^\frac{q}{2}K^*A^pKB^\frac{q}{2})^s\big) =&\ \phi\big(|XZZ^{-1}Y|^{2s}\big)\\
\leq&\ \phi\big(|XZ|^{2s_1}\big)^\frac{s}{s_1}\phi\big(|Z^{-1}Y|^{2s_2}\big)^\frac{s}{s_2}\\
\leq&\ \frac{s}{s_1}\phi\big((Z^*X^*XZ)^{s_1}\big) + \frac{s}{s_2} \phi\big((Y^*(Z^{-1})^*Z^{-1}Y)^{s_2}\big)\\
=&\ \frac{s}{s_1}\phi\big((Z^*X^*XZ)^{s_1}\big) + \frac{s}{s_2} \phi\big((Z^{-1}YY^*(Z^{-1})^*)^{s_2}\big). 
\end{align*}
We have used the fact that $\phi\big(f(|M|)\big) = \phi\big(f(|M^*|)\big)$ for any matrix $M\in\mathbb{C}^{n\times n}$ and any function $f$, since $\phi$ is only a function of eigenvalues and the spectrums of $f(|M|)$ and $f(|M^*|)$ are the same. Let $(XY)^* = Q|(XY)^*|$ be the polar decomposition of $(XY)^*$, where $Q\in\mathbb{C}^{n\times n}$ is unitary. So we have $XYQ = |(XY)^*|$. If $X$ and $Y$ are invertible, we can particularly choose $Z = YQ|(XY)^*|^{-\frac{s_1}{s_1+s_2}}$ to have
\[XZ = XYQ|(XY)^*|^{-\frac{s_1}{s_1+s_2}} = |(XY)^*|^{\frac{s_2}{s_1+s_2}},\]
\[\text{and}\quad Z^{-1}Y = |(XY)^*|^{\frac{s_1}{s_1+s_2}}Q^*,\]
which yields the equality 
\[\frac{s}{s_1}\phi\big((Z^*X^*XZ)^{s_1}\big) + \frac{s}{s_2} \phi\big((Z^{-1}YY^*(Z^{-1})^*)^{s_2}\big) = \phi\big(|(XY)^*|^{\frac{2s_1s_2}{s_1+s_2}}\big) = \phi\big(|XY|^{2s}\big).\]
Now for general $X,Y$ that are not necessarily invertible, we can always find two sequences of invertible matrices $\{X_j\}_{j=1}^{+\infty},\{Y_j\}_{j=1}^{+\infty}$ such that (i) $X_j\rightarrow X$, $Y_j\rightarrow Y$ and (ii) $X_j^*X_j\succeq X^*X$, $Y_jY_j^*\succeq YY^*$. Such sequences can be easily obtained by perturbing the singular values of $X$ and $Y$. For each pair of $(X_j,Y_j)$, we can find some invertible $Z_j$ so that the above equality holds. Also, for any invertible $Z$, we have $Z^*X_j^*X_jZ\succeq Z^*X^*XZ$, $Z^{-1}Y_jY_j^*(Z^{-1})^*\succeq Z^{-1}YY^*(Z^{-1})^*$, and thus 
\[\phi\big((Z^*X^*XZ)^{s_1}\big)\leq \phi\big((Z^*X_j^*X_jZ)^{s_1}\big), \quad \phi\big((Z^{-1}YY^*(Z^{-1})^*)^{s_2}\big)\leq \phi\big((Z^{-1}Y_jY_j^*(Z^{-1})^*)^{s_2}\big)\]
by \Cref{thm:Preserving} that we will prove in \Cref{sec:OtherResults}. Then we obtain a sequence of inequalities, 
\begin{align*}
\phi(|XY|^{2s}) \leq&\ \inf\{\frac{s}{s_1}\phi\big((Z^*X^*XZ)^{s_1}\big) + \frac{s}{s_2} \phi\big((Z^{-1}YY^*(Z^{-1})^*)^{s_2}\big): Z\ \text{invertible}\}\\
\leq&\ \frac{s}{s_1}\phi\big((Z_j^*X^*XZ_j)^{s_1}\big) + \frac{s}{s_2} \phi\big((Z_j^{-1}YY^*(Z_j^{-1})^*)^{s_2}\big)\\
\leq&\ \frac{s}{s_1}\phi\big((Z_j^*X_j^*X_jZ_j)^{s_1}\big) + \frac{s}{s_2} \phi\big((Z_j^{-1}Y_jY_j^*(Z_j^{-1})^*)^{s_2}\big)\\
=&\ \phi(|X_jY_j|^{2s}). 
\end{align*}
But since $\phi(|XY|^{2s}) = \lim_{j\rightarrow +\infty}\phi(|X_jY_j|^{2s})$ by continuity, the first inequality above must be an equality. Therefore, by substituting $X=A^\frac{p}{2},Y = KB^\frac{q}{2}$, we obtain that 
\[\phi\big((B^\frac{q}{2}K^*A^pKB^\frac{q}{2})^s\big) = \inf\{\frac{s}{s_1}\phi\big((Z^*A^pZ)^{s_1}\big) + \frac{s}{s_2} \phi\big((Z^{-1}KB^qK^*(Z^{-1})^*)^{s_2}\big): Z\ \text{invertible}\}.\]
Note that $s\in[0,\frac{1}{p+q}]$ implies $s_1\in[0,\frac{1}{p}],s_2\in[0,\frac{1}{q}]$. By \Cref{lem:GeneralEpstein}, the map 
\[(A,B) \longmapsto \frac{s}{s_1}\phi\big((Z^*A^pZ)^{s_1}\big) + \frac{s}{s_2} \phi\big((Z^{-1}KB^qK^*(Z^{-1})^*)^{s_2}\big)\]
is jointly concave in $(A,B)$ for every invertible $Z$, which then implies the joint concavity of the infimum over all invertible $Z$.

\end{proof}

\begin{proof} [\rm\textbf{Proof of \Cref{thm:GeneralLieb} (Part I)}]
We first prove the theorem for $m=1$. Let $r = p_1\in[0,1]$, and $K^{(N)} = (K^{(N)})^* = \exp\big(\frac{1}{2N}H\big),N\geq 1$. Then using the Lie product formula
\[\lim_{N\rightarrow+\infty}\left(\exp\big(\frac{1}{2N}Y\big)\exp\big(\frac{1}{N}X\big)\exp\big(\frac{1}{2N}Y\big)\right)^N=\exp(X+Y),\quad X,Y\in \mathbf{H}_n,\]
we have 
\begin{align*}
\lim_{N\rightarrow+\infty}\phi\Big(\big((K^{(N)})^*A^\frac{r}{N}K^{(N)}\big)^{N}\Big) =&\ \lim_{N\rightarrow+\infty}\phi\left(\Big(\exp\big(\frac{1}{2N}H\big)\exp\big(\frac{r}{N}\log A\big)\exp\big(\frac{1}{2N}H\big)\Big)^N\right)\\
=&\ \phi\big(\exp(H+r\log A)\big).
\end{align*}
By \Cref{thm:GeneralLiebConcavity}, for each $N\geq1$, $\phi\Big(\big((K^{(N)})^*A^\frac{r}{N}K^{(N)}\big)^N\Big)$ is concave in $A$, thus the limit function $\phi\big(\exp(H+r\log A)\big)$ is also concave in $A$.
\end{proof}

To go from $m=1$ to $m>1$ in \Cref{thm:GeneralLieb}, we need to use the convexity of the map $A\mapsto \phi(\exp(A))$, which we will prove via the following lemmas. They are the $k$-trace extensions of the Araki--Lieb--Thirring inequality \cite{araki1990inequality}, the Golden--Thompson inequality and a variant of the Peierls--Bogoliubov inequality (see e.g. \cite[Theorem 2.12]{carlen2010trace}). 

\begin{lemma}[$k$-trace Araki--Lieb--Thirring Inequality]
\label{lem:ktraceALT}
For any $A,B\in\mathbf{H}_n^+$, the function 
\[t\mapsto\trace_k\big[(B^\frac{t}{2}A^tB^\frac{t}{2})^\frac{1}{t}\big]\]
is monotone increasing on $(0,+\infty)$, that is 
\begin{equation}\label{eqt:ktraceALT}
\trace_k\big[(B^\frac{t}{2}A^tB^\frac{t}{2})^\frac{1}{t}\big]\leq \trace_k\big[(B^\frac{s}{2}A^sB^\frac{s}{2})^\frac{1}{s}\big],\quad 0<t\leq s.
\end{equation}
\end{lemma}
\begin{proof}
Using the definition and properties of the operator $\MM_0^{(k)}$ in \Cref{Apdix:ExteriorAlgebra}, we have that
\begin{align*}
\trace_k\big[(B^\frac{t}{2}A^tB^\frac{t}{2})^\frac{1}{t}\big] =&\  \trace\big[\MM_0^{(k)}\big((B^\frac{t}{2}A^tB^\frac{t}{2})^\frac{1}{t}\big)\big]\\
=&\ \trace\big[\big((\MM_0^{(k)}(B))^\frac{t}{2}(\MM_0^{(k)}(A))^t(\MM_0^{(k)}(B))^\frac{t}{2}\big)^\frac{1}{t}\big].
\end{align*}
Since $A,B\in\mathbf{H}_n^+$, $\MM_0^{(k)}(A)$ and $\MM_0^{(k)}(B)$ are both Hermitian and positive semidefinite. Then inequality \eqref{eqt:ktraceALT} follows immediately from the original Araki--Lieb--Thirring inequality \cite{araki1990inequality} for normal trace. 
\end{proof}

\begin{lemma}[$k$-trace Golden--Thompson Inequality]
\label{lem:ktraceGT}
For any $A,B\in \mathbf{H}_n$, 
\begin{equation}
\trace_k\big[\exp(A+B)\big]\leq \trace_k\big[\exp(A)\exp(B)\big],
\end{equation}
with equality holds if and only if $AB=BA$.
\end{lemma}

\begin{proof}
We here only prove the inequality. The condition for equality will be justified in an alternative proof of this lemma in \Cref{Apdix:ExteriorAlgebra}. For any $A,B\in \mathbf{H}_n$, we have
\begin{align*}
\trace_k\big[\exp(A+B)\big] =&\ \lim_{m\rightarrow +\infty}\trace_k\Big[\Big(\exp\big(\frac{1}{2m}B\big)\exp\big(\frac{1}{m}A\big)\exp\big(\frac{1}{2m}B\big)\Big)^m\Big]\\
\leq&\ \trace_k\big[\exp\big(\frac{1}{2}B\big)\exp\big(A\big)\exp\big(\frac{1}{2}B\big)\big]\\
=&\ \trace_k\big[\exp\big(A\big)\exp\big(B\big)\big].
\end{align*}
The first equality above is the Lie product formula, and the inequality is due to \Cref{lem:ktraceALT}.
\end{proof}

\begin{lemma}[$k$-trace Peierls--Bogoliubov Inequality]
\label{lem:ktracePB}
The function 
\begin{equation}
A \ \longmapsto\ \log\trace_k\big[\exp(A)\big]
\end{equation}
is convex on $\mathbf{H}_n$.
\end{lemma}

\begin{proof}
For any $A,B\in \mathbf{H}_n$, $\tau\in(0,1)$, by \Cref{lem:ktraceGT} we have
\begin{align*}
\trace_k\big[\exp(\tau A+(1-\tau) B)\big] \leq&\ \trace_k\big[\exp\big(\tau A\big)\exp\big((1-\tau)B\big)\big]\\
\leq&\ \trace_k\big[\exp(A)\big]^\tau \trace_k\big[\exp(B)\big]^{1-\tau}.
\end{align*}
The second inequality above is H\"older's. Therefore 
\[\log\trace_k\big[\exp(\tau A+(1-\tau)B)\big]\leq \tau\log\trace_k\big[\exp(A)\big]+(1-\tau)\log\trace_k\big[\exp(B)\big].\]
\end{proof}

We remark that \Cref{lem:ktracePB} can also be proved using the operator interpolation in \Cref{lem:KeyLemma}. \Cref{lem:ktracePB} immediately implies that $A\mapsto \log\phi\big(\exp(A)\big)=\frac{1}{k}\log\trace_k\big[\exp(A)\big]$ is convex, and thus $A\mapsto\phi\big(\exp(A)\big)$ is convex. This will help us prove improve from $m=1$ to $m\geq 1$ in \Cref{thm:GeneralLieb}.  

\begin{proof}[\rm\textbf{Proof of \Cref{thm:GeneralLieb} (Part II)}] Given any $\{A^{(j)}\}_{j=1}^m,\{B^{(j)}\}_{j=1}^m\subset\mathbf{H}_n^{++}$, and any $\tau\in[0,1]$, let $C^{(j)} = \tau A^{(j)}+(1-\tau)B^{(j)},1\leq j\leq m$. Since the map $X\mapsto\phi(\exp(X))$ is convex on $\mathbf{H}_n$, the map $X\mapsto\phi(\exp(L+X))$ is also convex on $\mathbf{H}_n$ for arbitrary $L\in\mathbf{H}_n$. Now define 
\[L= H +\sum_{j=1}^mp_j\log C^{(j)},\quad r=\sum_{j=1}^mp_j\leq 1.\]
If $r=0$, there is nothing to prove, so we may assume that $r>0$. We then have that
\begin{align*}
\phi\big(\exp(H+\sum_{j=1}^mp_j\log X^{(j)})\big) =&\ \phi\Big(\exp\big(H+r\sum_{j=1}^m\frac{p_j}{r}(\log X^{(j)}-\log C^{(j)})+ \sum_{j=1}^mp_j\log C^{(j)}\big)\Big)\\
=&\ \phi\Big(\exp\big(L+r\sum_{j=1}^m\frac{p_j}{r}(\log X^{(j)}-\log C^{(j)})\big)\Big)\\
\leq&\ \sum_{j=1}^m\frac{p_j}{r}\phi\big(\exp(L+r\log X^{(j)}-r\log C^{(j)})\big),\quad X^{(j)}=A^{(j)},B^{(j)}.
\end{align*}
For each $j$, by the concavity of \eqref{eqt:function3} for $m=1$, we have 
\begin{align*}
&\ \tau \phi\big(\exp(L+r\log A^{(j)}-r\log C^{(j)})\big) + (1-\tau)\phi\big(\exp(L+r\log B^{(j)}-r\log C^{(j)})\big)\\
\leq &\ \phi\big(\exp(L+r\log(\tau A^{(j)}+(1-\tau) B^{(j)})-r\log C^{(j)})\big) \\
=&\ \phi\big(\exp(L)\big).
\end{align*}
Therefore we obtain that 
\begin{align*}
&\ \tau \phi\big(\exp(H+\sum_{j=1}^mp_j\log A^{(j)})\big) + (1-\tau)\phi\big(\exp(H+\sum_{j=1}^mp_j\log B^{(j)})\big)\\
\leq &\ \sum_{j=1}^m\frac{p_j}{r} \phi\big(\exp(L)\big)\\
=&\ \phi\big(\exp(H+\sum_{j=1}^mp_j\log C^{(j)})\big),
\end{align*}
that is, \eqref{eqt:function3} is jointly concave on $(\mathbf{H}_n^{++})^{\times m}$ for all $m\geq 1$.
\end{proof}

\subsection{Some Corollaries}

The following corollary follows from standard arguments on homogeneous, concave functions.

\begin{corollary}\label{cor:HomogeneousConcave}
For any $p,q\in[0,1],p+q>0,s\in(0,\frac{1}{p+q}]$, and any $K\in \mathbb{C}^{n\times m}$, the function
\begin{equation}
(A,B) \ \longmapsto\ \phi\big((B^\frac{q}{2}K^*A^pKB^\frac{q}{2})^s\big)^\frac{1}{s(p+q)}
\label{eqt:function4}
\end{equation}
is jointly concave on $\mathbf{H}^{+}_n\times\mathbf{H}_m^+$. For any $H\in \mathbf{H}_n$ and any $\{p_j\}_{j=1}^m\subset[0,1]$ such that $0<\sum_{j=1}^mp_j\leq1$, the function 
\begin{equation}
(A^{(1)},A^{(2)},\dots,A^{(m)}) \ \longmapsto\ \phi\big(\exp\big(H+\sum_{j=1}^mp_j\log A^{(j)}\big)\big)^\frac{1}{\sum_{j=1}^mp_j},
\label{eqt:function5}
\end{equation}
is jointly concave on $(\mathbf{H}^{++}_n)^{\times m}$.
\end{corollary}

\begin{proof} Consider any matrix function $F:\mathbf{H}_n^+\rightarrow [0,+\infty)$(or $\mathbf{H}_n^{++}\rightarrow [0,+\infty)$) that is positively homogeneous of order 1, i.e. $F(\lambda A)=\lambda F(A),\forall \lambda\geq 0$. By \Cref{lem:homogeneous}, we have that $F$ is concave $\Longleftrightarrow$ $F^r$ is concave for some $r\in(0,1]$. 

One can easily check that the functions (\ref{eqt:function4}) and (\ref{eqt:function5}) are positively homogeneous of order 1. Then this corollary follows from \Cref{thm:GeneralLiebConcavity}, \Cref{thm:GeneralLieb} and \Cref{lem:homogeneous}.
\end{proof}

The following corollary is an analog of the concave part of Lemma 3.1 in \cite{Carlen2008}.

\begin{corollary}\label{cor:MultiConcave}
For any $r\in[0,1],s\in[0,\frac{1}{r}]$ and any $\{K^{(j)}\}_{j=1}^m\subset \mathbb{C}^{n\times n}$, the function 
\begin{equation}
(A^{(1)},A^{(2)},\dots,A^{(m)}) \ \longmapsto\  \phi\Big(\Big(\sum_{j=1}^m(K^{(j)})^*(A^{(j)})^rK^{(j)}\Big)^s\Big)
\label{eqt:function6}
\end{equation}
is jointly concave on $(\mathbf{H}_n^+)^{\times m}$.
\end{corollary}
\begin{proof}
Define 
\[\widehat{A} = \left(\begin{array}{cccc} 
A^{(1)} & \bm{0} & \dots & \bm{0} \\ 
\bm{0} & A^{(1)} & \dots & \bm{0} \\
\vdots & \vdots & \ddots & \vdots\\
\bm{0} & \bm{0} & \dots & A^{(m)} 
\end{array}\right)\in \mathbf{H}_{mn}^{+},\quad  
 \widehat{K} = \left(\begin{array}{cccc} 
 K^{(1)} & \bm{0} & \dots & \bm{0} \\ 
 K^{(2)} & \bm{0} & \dots & \bm{0} \\
 \vdots & \vdots & \ddots & \vdots \\
 K^{(m)} & \bm{0} & \dots & \bm{0}  \end{array}\right)\in\mathbb{C}^{mn\times mn}.\]
 Then we have
 \[(\widehat{K}^*\widehat{A}^r\widehat{K})^s = 
 \left(\begin{array}{ccc} 
 \Big(\sum_{j=1}^m(K^{(j)})^*(A^{(j)})^rK^{(j)}\Big)^s & \dots & \bm{0} \\ 
\vdots & \ddots & \vdots\\
\bm{0} & \dots & \bm{0}
\end{array}\right),\]
and thus 
\[\phi\big((\widehat{K}^*\widehat{A}^r\widehat{K})^s\big) = \phi\Big(\Big(\sum_{j=1}^m(K^{(j)})^*(A^{(j)})^rK^{(j)}\Big)^s\Big).\]
By \Cref{lem:GeneralEpstein}, the left hand side above is concave in $\widehat{A}$, therefore the right hand side is jointly concave in $(A^{(1)},A^{(2)},\dots,A^{(m)})$.
\end{proof}

\subsection{Revisiting previous proofs in the trace case}
In this section, we will review some previous works on concavity results of trace functions. The purpose is to compare by example the spirits of methods from different perspectives, so as to explain why we have chosen the interpolation technique by Stein and the variational method by Zhang to prove our main results. For a whole story of known results on both convexity and concavity, one may refer to \cite{LIEB1973267,ANDO1979203,carlen2010trace,carlen2018inequalities,ZHANG2020107053}. As mentioned in the introduction, many alternative proofs of Lieb's concavity theorem (the concavity of function \eqref{eqt:LCT}) have been found since its original establishment by Lieb in 1973. A proof using matrix tensors was given by Ando \cite{ANDO1979203} in 1979 (see also Carlen \cite{carlen2010trace}). Ando interpreted $\trace[K^*A^pKB^q]$ as an inner product on the tensor space $\mathbb{C}^n\otimes \mathbb{C}^m$ and translated the Lieb's concavity theorem to the statement that the map $(A,B)\mapsto A^{p}\otimes B^{q}$ is operator concave. Ando then proved the latter using the integral representation of $A^p$ (see below). Here $\otimes$ is the Kronecker product. Later, Nikoufar et al. \cite{NIKOUFAR2013531} provided a simpler proof for the concavity of $(A,B)\mapsto A^{p}\otimes B^{q}$ using the concept of matrix perspectives (see e.g. \cite{ebadian2011perspectives}). We summarize the ideas of their proofs as follows. For simplicity, we assume that $p+q=1$. The result for $p+q=r<1$ can be further obtained by using the fact that $A\mapsto A^r$ is operator monotone increasing and operator concave for $r\in[0,1]$. For $p\in[0,1]$, the map $A\mapsto (A\otimes I_m)^p=A^p\otimes I_m$ from $\mathbf{H}_n^+$ to $\mathbf{H}_{nm}^+$ is operator concave, and thus its perspective from $\mathbf{H}_n^+\times \mathbf{H}_m^+$ to $\mathbf{H}_{nm}^+$, 
\[(A,B) \mapsto (I_n\otimes B)^\frac{1}{2}\big((I_n\otimes B)^{-\frac{1}{2}}(A\otimes I_m)(I_n\otimes B)^{-\frac{1}{2}}\big)^p(I_n\otimes B)^\frac{1}{2} = A^p\otimes B^{1-p},\]
is jointly operator concave in $(A,B)$. The simplified expression above results from the fact that $A\otimes I_m$ commutes with $I_n\otimes B$. For any $K\in\mathbb{C}^{n\times m}$, we have the identity (a variant of Ando's interpretation)   
\begin{equation}\label{eqt:AndoIdentity}
\trace[K^*A^pKB^{1-p}] = \Big\langle \sum_{j=1}^m(Ke_j^{(m)})\otimes e_j^{(m)}, A^p\otimes (B^T)^{1-p}\sum_{j=1}^m(Ke_j^{(m)})\otimes e_j^{(m)}\Big\rangle_{\mathbb{C}^n\otimes \mathbb{C}^m},
\end{equation}
where $B^T$ is the transpose of $B$, and $e^{(m)}_j=(0,\dots,\overset{j_{\text{th}}}{1},\dots,0)\in\mathbb{C}^m$. Note that since $B\in\mathbf{H}_n^+$, $B^T$ is also in $\mathbf{H}_n^+$. Since $B\mapsto B^T$ is linear, the joint operator concavity of $(A,B)\mapsto A^p\otimes B^{1-p}$ then implies the joint concavity of $(A,B)\mapsto\trace[K^*A^pKB^{1-p}]$. 

As an application, Carlen and Lieb \cite{Carlen2008} applied the Lieb's concavity theorem to prove the concavity of $A\mapsto \trace[(K^*A^rK)^s]$ for $r\in[0,1],s\in[1,\frac{1}{r}]$ (they used a slightly different but equivalent expression) based on a variational characterization of this function (the supremum part of \cite[Lemma 2.2]{Carlen2008}). We here provide a simplified proof that captures the main spirit. For any $A,B\in \mathbf{H}_n^+$, $K\in\mathbb{C}^{n\times n}$, let
\[X = (K^*A^rK)^s,\quad Y = (K^*B^rK)^s.\]
Then for any $\tau\in[0,1]$, note that $\frac{1}{s}\leq 1, r+(1-\frac{1}{s})\leq 1$, we have
\begin{equation}\label{eqt:SimpleProof}
\begin{split}
\tau\trace[X] + (1-\tau)\trace[Y] =&\  \tau\trace\big[K^*A^rKX^{1-\frac{1}{s}}\big] + (1-\tau)\trace\big[K^*B^rKY^{1-\frac{1}{s}}\big] \\ 
\leq&\ \trace\big[K^*(\tau A+(1-\tau)B)^rK(\tau X+(1-\tau)Y)^{1-\frac{1}{s}}\big] \\
\leq&\ \trace\big[\big(K^*(\tau A+(1-\tau)B)^rK\big)^s\big]^\frac{1}{s}\trace\big[\tau X+(1-\tau)Y\big]^{1-\frac{1}{s}} \\
=&\ \trace\big[\big(K^*(\tau A+(1-\tau)B)^rK\big)^s\big]^\frac{1}{s}\big(\tau\trace[X] + (1-\tau)\trace[Y]\big)^{1-\frac{1}{s}}, 
\end{split}
\end{equation}
where the first inequality is due to Lieb's concavity theorem with $p=r,q=1-\frac{1}{s},p+q=r+1-\frac{1}{s}\leq 1$, and the second inequality is H\"older's. The above then simplifies to 
\[\tau\trace\big[(K^*A^rK)^s] + (1-\tau)\trace[(K^*B^rK)^s\big]\leq\trace\big[\big(K^*(\tau A+(1-\tau)B)^rK\big)^s\big],\]
which concludes the concavity of $A\mapsto \trace[(K^*A^rK)^s]$.

The variational characterizations of $\trace[(K^*A^rK)^s]$ in \cite{Carlen2008} can be abstracted to the following two formulas (\cite[Lemma 12]{carlen2018inequalities}): for any $X\in \mathbf{H}_n^+$,
\begin{align}
&\trace[X^s] = \sup\left\{s\trace[XY] - (s-1) \trace[Y^\frac{s}{s-1}]:\ Y\in \mathbf{H}_n^{+}\right\} \quad \text{if $s>1$ or $s<0$};\label{eqt:Carlen_sup}\\
&\trace[X^s] = \inf\left\{s\trace[XY] + (1-s)\trace[Y^{-\frac{s}{1-s}}]:\ Y\in \mathbf{H}_n^{++}\right\} \quad \text{if $0<s<1$}.\label{eqt:Carlen_inf}
\end{align}
Further, these variational formulas were used to derive the convexity/concavity of the function $(A,B)\mapsto \trace[(B^\frac{q}{2}K^*A^pKB^\frac{q}{2})^s]$ for a partial range of $p,q,s$ (partial to the necessary conditions on $p,q,s$ for the corresponding convexity/concavity to hold). For example, formula \eqref{eqt:Carlen_inf} was used by Carlen et al. to prove the concavity for $0\leq p,q\leq 1, 0<s\leq \frac{1}{1+q}$ \cite[Theorem 4.4]{carlen2016some}. Recently, the above formulas were modified by Zhang to the following \cite[Theorem 3.3]{ZHANG2020107053}: for any $X,Y\in \mathbb{C}^{n\times n}$ and any $r_0,r_1,r_2>0$ such that $\frac{1}{r_0}=\frac{1}{r_1}+\frac{1}{r_2}$, 
\begin{align*}
&\trace[|XY|^{r_1}] = \sup\left\{\frac{r_1}{r_0}\trace[|XZ|^{r_0}] - \frac{r_1}{r_2}\trace[|Y^{-1}Z|^{r_2}]:\ Z\in\mathbb{C}^{n\times n}\right\}\quad \text{if $Y$ is invertible};\\
&\trace[|XY|^{r_0}] = \inf\left\{\frac{r_0}{r_1}\trace[|XZ|^{r_1}] + \frac{r_0}{r_2}\trace[|Z^{-1}Y|^{r_2}]:\ Z\in\mathbb{C}^{n\times n}\ \text{invertible}\right\}.
\end{align*}
Zhang then used them to provide a unified variational proof of the joint convexity/concavity of $(A,B)\mapsto \trace[(B^\frac{q}{2}K^*A^pKB^\frac{q}{2})^s]$ for the full range of $p,q,s$, finally confirming that the sufficient conditions on $p,q,s$ coincide with the necessary conditions. 

These arguments using matrix tensors and variational forms were also adopted by Tropp \cite{MAL-048} to provide an alternative proof of the concavity of $A\mapsto \trace[\exp(H+\log A)]$. Tropp's proof is based on his variational formula for trace, 
\begin{equation}\label{eqt:TroppFormula}
\trace[M] = \sup_{T\in\mathbf{H}_n^{++}}\trace\big[T\log M-T\log T +T\big], \quad M\in\mathbf{H}_n^{++},
\end{equation}
which relies on the non-negativeness of the matrix relative entropy
\[D(T;M) = \trace\big[T(\log T-\log M) - (T-M)\big],\quad T,M\in \mathbf{H}_n^{++}.\]
The non-negativeness of $D(T;M)$ is a classical result of Klein's inequality (see Petz \cite[Proposition 3]{petz1994survey}, Carlen \cite[Theorem 2.11]{carlen2010trace} or Tropp \cite[Proposition 8.3.5]{MAL-048}). Tropp substituted $M=\exp(H+\log A)$ in \eqref{eqt:TroppFormula} to obtain
\[\trace[\exp(H+\log A)] = \sup_{T\in \mathbf{H}_n^{++}}\big(\trace[TH]+\trace[A]-D(T;A)\big).\]
The concavity of $A\mapsto \trace[\exp(H+\log A)]$ then follows from this variational expression, the joint convexity of $D(T;A)$ in $(T,A)$, and the fact that $g(x) = \sup_{y\in\Omega}f(x,y)$ is concave in $x$ if $f(x,y)$ is jointly concave in $(x,y)$ and $\Omega$ is convex (see e.g. \cite[Lemma 2.3]{Carlen2008}). The joint convexity of the relative entropy $D(T;A)$ was first due to Lindblad \cite{lindblad1973entropy}. One can also see Ando \cite{ANDO1979203}, Carlen \cite{carlen2010trace} and Tropp \cite{MAL-048} for alternative proofs. 

A Methodology of another flavor arose from the use of complex analysis. In the same year of Lieb's original paper on his concavity theorem, Epstein \cite{epstein1973remarks} provided a unified way of proving the concavity of $A\mapsto \trace[K^*A^pKA^q]$, $A\mapsto \trace[(K^*A^sK)^\frac{1}{s}]$ and $A\mapsto \trace[\exp(H+\log A)]$, using a derivative argument based on the theory of Herglotz functions (functions that are analytic in the open upper half plane $\mathbb{C}_+$ and have a positive imaginary part). Epstein's method relies on integral representations of matrix powers, which has a deep connection to a profound theorem of Loewner's : a real-valued function $f$ on $(0,+\infty)$ is operator monotone if and only if it admits an analytic continuation to a Herglotz function. Loewner's theorem provides a convenient tool for understanding trace functions by passing the study of a desired property from the integral to the integrand that has a relatively simpler form. One may see the book of Donoghue \cite{Donoghue1974MonotoneMF} for a full account of this theory. Epstein's approach was further developed by Hiai \cite{hiai2013concavity,hiai2016concavity} and, for example, adopted in his first proof of the concavity of $(A,B)\mapsto \trace[(B^\frac{q}{2}K^*A^pKB^\frac{q}{2})^s]$ for the full range $0\leq p,q\leq 1,0\leq s\leq \frac{1}{p+q}$ \cite[Theorem 2.1]{hiai2016concavity}. Specifically, by substituting $\sigma = s(p+q)<1$, $X = (B^\frac{q}{2}K^*A^pKB^\frac{q}{2})^{\frac{1}{p+q}}$ in the integral formula, 
\begin{equation}\label{eqt:integral_formula}
X^\sigma = \frac{\sin(\pi \sigma)}{\pi}\int_0^\infty t^{-1+\sigma}(I_n + tX^{-1})^{-1}dt, \quad 0<\sigma<1,\ X\in \mathbf{H}_n^{++},
\end{equation}
Hiai passed the joint concavity of $(A,B)\mapsto\trace[(B^\frac{q}{2}K^*A^pKB^\frac{q}{2})^s]$ to the joint concavity of $(A,B)\mapsto\trace\Big[1 + t(B^\frac{q}{2}K^*A^pKB^\frac{q}{2})^{-\frac{1}{p+q}}\Big]^{-1}$ (the case $s(p+q)=1$ was handled differently by directly taking $t\rightarrow +\infty$). He then proved the latter using the derivative argument introduced by Epstein.

The introduction of complex analysis into the field has also led to another branch of methods based on interpolation theories. In his original proof of the concavity of $(A,B)\mapsto \trace[K^*A^pKB^q]$ for $0\leq p,q\leq 1,p+q\leq 1$ \cite[Theorem 1]{LIEB1973267}, Lieb made use of the maximum modulus principle to concentrate the powers of $A^p,B^q$ to the only power $p+q$ on $A$ or $B$, and then proceeded with the operator concavity of $X\mapsto X^{p+q}$. This technique, relying on the holomorphicity of $X^z$ ($X\in\mathbf{H}_n^+$) as a function of $z$, already shed some light on the use of complex interpolation theories. Later, Uhlmann \cite{uhlmann1977relative} applied interpolation theories explicitly to again prove Lieb's concavity theorem, by interpreting $\trace[K^*A^pKB^{1-p}]$ as an interpolation between $\trace[K^*AK]$ and $\trace[K^*KB]$. Uhlmann's quadratic interpolation of seminorms extended the relevant works of Lieb to positive linear forms of arbitrary *-algebras. Kosaki \cite{kosaki1982interpolation} further explored the idea of quadratic interpolation of seminorms and captured Lieb's concavity theorem in the frame of general interpolation theories.

The above methodologies all have a unique perspective in understanding complicated trace functions. However, some of them are found hardly generalizable to $k$-trace, as they more or less rely on the linearity of the normal trace. For example, Ando's identity \eqref{eqt:AndoIdentity} and Tropp's variational formula \eqref{eqt:TroppFormula}. Our $k$-trace function $\phi$ for $k>1$ is at best sub-additive since it is concave and homogeneous of order 1. Hiai's use of the integral formula also lives on linearity in that the trace operation can be pulled into the integral. Though we can interpret the $k$-trace as the trace of an antisymmetric tensor, the power of $\frac{1}{k}$ will have to stay out of the integral, and the antisymmetric tensor in the integrand will also bring huge difficulties to the derivative argument that comes after. 

The variational method introduced by Carlen and Lieb needs the linearity of trace as well in proving the concavity of $A\mapsto \trace[(K^*A^rK)^s]$. One can see this in the last equality of \eqref{eqt:SimpleProof}, which will become an inequality in the undesirable direction if $\trace[\cdot]$ is replaced by $\trace_k[\cdot]^\frac{1}{k}$, since $\trace_k[X]^\frac{1}{k}$ is concave in $X$. In fact, the argument in \eqref{eqt:SimpleProof} that reduces the concavity of $A\mapsto \trace[(K^*A^rK)^s]$ for $r\in[0,1],s\in[1,\frac{1}{r}]$ to the joint concavity of $(A,B)\mapsto \trace[K^*A^pKB^q]$ for $p,q\in[0,1],p+q\leq 1$ was performed in a variational manner in \cite{Carlen2008} based on the following formula (which is equivalent to the expression of the supremum case in \cite[Lemma 2.2]{Carlen2008}):
\begin{equation}\label{eqt:variational_CarlenLieb}
\trace[(K^*A^rK)^s] = \sup_{X\in\mathbf{H}_n^{++}}\left\{s\trace[K^*A^rKX^{1-\frac{1}{s}}] - (s-1)\trace[X]\right\} = \sup_{X\in\mathbf{H}_n^{++}} \Psi(A,X),
\end{equation}
which can be derived using H\"older's inequality for trace. Recall that $g(x) = \sup_{y\in\Omega}f(x,y)$ is concave in $x$ if $f(x,y)$ is jointly concave in $(x,y)$ and $\Omega$ is convex. Since $\trace[K^*A^rKX^{1-\frac{1}{s}}]$ is jointly concave in $(A,X)$ (as $r+(1-\frac{1}{s})\leq 1$) and $\trace[X]$ is linear in $X$, the function $\Psi(A,X)$ is jointly concave in $(A,X)$. The concavity of $A\mapsto \trace[(K^*A^rK)^s]$ then follows from \eqref{eqt:variational_CarlenLieb}. It is then natural to consider a similar variational formula for the $k$-trace that can also be shown by H\"older's inequality:
\begin{equation}\label{eqt:variational_CarlenLieb_ktrace}
\trace_k[(K^*A^rK)^s]^\frac{1}{k} = \sup_{X\in\mathbf{H}_n^{++}}\left\{s\trace_k[K^*A^rKX^{1-\frac{1}{s}}]^\frac{1}{k} - (s-1)\trace_k[X]^\frac{1}{k}\right\} = \sup_{X\in\mathbf{H}_n^{++}} \Psi_k(A,X).
\end{equation}
Note that for $s>1$, we have $-(s-1)<0$ and thus $- (s-1)\trace_k[X]^\frac{1}{k}$ is convex in $X$ since $\trace_k[\cdot]^\frac{1}{k}$ is concave (this sign of $-(s-1)$ does not give trouble in the trace case due to the linearity of trace). Therefore, even provided that $(A,X)\mapsto \trace_k[K^*A^rKX^{1-\frac{1}{s}}]^\frac{1}{k}$ is jointly concave, the function $\Psi_k(A,X)$ is not guaranteed to be jointly concave in $(A,X)$, and the variational formula in \eqref{eqt:variational_CarlenLieb_ktrace} fails to conclude the concavity of $A\mapsto \trace_k[(K^*A^rK)^s]^\frac{1}{k}$. As a consequence, we have not found a way to adapt this particular argument into a proof of \Cref{lem:GeneralEpstein}.

However, these variational approaches, especially Zhang's variational characterizations, are conveniently applicable to the derivation of the more general cases given \Cref{lem:GeneralEpstein}, as we have seen in the proof of \Cref{thm:GeneralLiebConcavity}. One can see that the $k$-trace version of H\"older's inequality plays an essential role in the process, which has suggested us to employ complex interpolation theories in the first place as interpolation of operators is based essentially on H\"older's inequality. In particular, we found the operator interpolation technique (\Cref{thm:SHInterpolation}) developed by Stein \cite{stein1956interpolation}(1956) nicely compatible to our problem. One can derive a variety of interpolation inequalities systematically by choosing $G(z)$ properly in inequality \eqref{eqt:SHInterpolation}. The choice of $G(z)$ we make, inspired by Lieb's original construction, gives the proof of our \Cref{lem:GeneralEpstein}. We remark that, \Cref{lem:KeyLemma} does not only apply to trace or $k$-trace, but also to any continuous matrix function $\phi:\mathbf{H}_n^+\rightarrow [0,+\infty)$ that satisfies H\"older's inequality and is unitary invariant in the sense that $\phi(U^*XU) = \phi(X)$ for arbitrary $X\in \mathbf{H}_n^+$ and $U\in \mathbb{C}^{n\times n}$ unitary. In fact, these two properties suffice to imply that $\log\circ \phi\circ \exp$ is convex on $\mathbf{H}_n$, which, along with a canonical majorization argument, will yield the inequality \eqref{eqt:KeyIneq} (see e.g. \cite{hiai2017generalized}).

\section{Other Results on $k$-trace}
\label{sec:OtherResults}

\subsection{Multivariate Golden--Thompson Inequality}

Sutter et al. \cite{sutter2017multivariate} recently applied the operator interpolation in \Cref{thm:SHInterpolation} to derive a multivariate extension of the Golden--Thompson (GT) inequality, which covers the original GT inequality and its three-matrix extension by Lieb \cite{LIEB1973267}. 

Following the ideas in \cite{sutter2017multivariate}, we may also use \Cref{lem:KeyLemma} to further extend the multivariate GT inequality to a $k$-trace form. In what follows, we write $\prod_{j=1}^mX^{(j)}$ for the matrix multiplication in the index order, i.e. $\prod_{j=1}^mX^{(j)}=X^{(1)}X^{(2)}\cdots X^{(m)}$. We first present an analog of Theorem 3.2 in \cite{sutter2017multivariate}. 

\begin{lemma}\label{lem:ktraceSBT}
For any $A^{(1)},A^{(2)},\dots,A^{(m)}\in \mathbf{H}_n^+$, $p\in[1,+\infty)$, $r\in(0,1]$, the following inequality holds:
\begin{equation}\label{eqt:ktraceSBT}
\log\phi\Big(\Big|\prod_{j=1}^m(A^{(j)})^r\Big|^\frac{p}{r}\Big) \leq \int_{-\infty}^{+\infty}dt\beta_r(t)\log\phi\Big(\Big|\prod_{j=1}^m(A^{(j)})^{1+it}\Big|^p\Big).
\end{equation}
\end{lemma}

\begin{proof} 
Define 
\[G(z) = \prod_{j=1}^m(A^{(j)})^z,\quad z\in\mathcal{S},\]
where $\mathcal{S}$ is defined as in \Cref{lem:KeyLemma}. One can check that $G(z)$ is holomorphic in the interior of $\mathcal{S}$ and continuous on the boundary, and $\|G(z)\|$ is uniformly bounded on $\mathcal{S}$. We may first assume that each $A^{(j)}\in\mathbf{H}_n^{++}$ so that $(A^{(j)})^{it},t\in\mathbb{R}$ is unitary. The result for $A^{(j)}\in\mathbf{H}_n^+$ can be obtained by continuity. Thus $G(it)$ is unitary for all $t\in\mathbb{R}$, and $|G(it)|^{p_0}=I_n$ for all $p_0$. Thus we can apply inequality \eqref{eqt:KeyIneq2} with $\theta=r,p_0\rightarrow +\infty,p_1=p,p_\theta=\frac{p}{r}$ to obtain
\[\log\phi\big(|G(r)|^\frac{p}{r}\big)\leq \int_{-\infty}^{+\infty}dt\beta_r(t)\log\phi\big(|G(1+it)|^p\big),\]
which is exactly \eqref{eqt:ktraceSBT}. 
\end{proof}

Using a multivariate version of the Lie product formula, we immediately obtain the following from \Cref{lem:ktraceSBT}.

\begin{thm}[Multivariate Golden--Thompson Inequality for k-trace]\label{thm:ktraceMultiGT}
For any $A^{(1)},A^{(2)},\dots,A^{(m)}\in \mathbf{H}_n$, the following inequality holds:
\begin{equation}\label{eqt:ktraceMultiGT}
\log\phi\Big(\Big(\exp\big(\sum_{j=1}^mA^{(j)}\big)\Big)^p\Big)\leq \int_{-\infty}^{+\infty}dt\beta_0(t)\log\phi\Big(\Big|\prod_{j=1}^m\exp\big((1+it)A^{(j)}\big)\Big|^p\Big).
\end{equation}
\end{thm}

\begin{proof}
We only need to replace $A^{(j)}$ in inequality \eqref{eqt:ktraceSBT} by $\exp(A^{(j)})$, and take $r\rightarrow0$. Since each $\|\exp((1+it)A^{(j)})\|=\|\exp(A^{(j)})\exp(itA^{(j)})\|=\|\exp(A^{(j)})\|$ is uniformly bounded for all $t\in\mathbb{R}$, the right hand side of \eqref{eqt:ktraceSBT} then becomes the right hand side of \eqref{eqt:ktraceMultiGT}. By a multivariate Lie product formula (see e.g. \cite{sutter2017multivariate})
\[\lim_{r\searrow0}\Big(\exp(rX^{(1)})\exp(rX^{(2)})\cdots\exp(rX^{(m)})\Big)^\frac{1}{r}=\exp\big(\sum_{i=1}^mX^{(j)}\big),\]
the left hand side of \eqref{eqt:ktraceSBT} then becomes 
\begin{align*}
\lim_{r\searrow0}\log\phi\Big(\Big|\prod_{j=1}^m\exp\big(rA^{(j)}\big)\Big|^\frac{p}{r}\Big) =&\ \lim_{r\searrow0}\log\phi\Big(\Big(\prod_{j=1}^m\exp\big(rA^{(m-j+1)}\big)\prod_{j=1}^m\exp\big(rA^{(j)}\big)\Big)^\frac{p}{2r}\Big)\\
=&\ \log\phi\Big(\Big(\exp\big(\sum_{j=1}^m2A^{(j)}\big)\Big)^\frac{p}{2}\Big)\\
=&\ \log\phi\Big(\Big(\exp\big(\sum_{j=1}^mA^{(j)}\big)\Big)^p\Big).
\end{align*}{}
\end{proof}

If we choose $m = 2,p=2$ in \Cref{thm:ktraceMultiGT} and replace $A^{(j)}$ by $\frac{1}{2}A^{(j)}$, the right hand side of inequality \eqref{eqt:ktraceMultiGT} is independent of $t$ due to the cyclicity of $\phi$. We then recover the $k$-trace GT inequality 
\[\phi\big(\exp(A^{(1)}+A^{(2)})\big)\leq \phi\big(\exp(A^{(1)})\exp(A^{(2)})\big),\]
that we have obtained in \Cref{lem:ktraceGT}. If we choose $m = 3,p=2$ in \Cref{thm:ktraceMultiGT} and again replace $A^{(j)}$ by $\frac{1}{2}A^{(j)}$, we have 
\begin{align*}
&\ \log\phi\big(\exp(A^{(1)}+A^{(2)}+A^{(3)})\big)\\
\leq&\ \int_{-\infty}^{+\infty}dt\beta_0(t)\log\phi\Big( \exp(A^{(1)})\exp\big(\frac{1+it}{2}A^{(2)}\big)\exp(A^{(3)})\exp\big(\frac{1-it}{2}A^{(2)}\big) \Big)\\
\leq&\ \log\phi\left( \int_{-\infty}^{+\infty}dt\beta_0(t)\exp(A^{(1)})\exp\big(\frac{1+it}{2}A^{(2)}\big)\exp(A^{(3)})\exp\big(\frac{1-it}{2}A^{(2)}\big) \right)
\end{align*}
The second inequality above is due to concavity of logarithm and $\phi$. If we define 
\[\mathcal{T}_A[B]=\int_0^{+\infty}dt(A+tI_n)^{-1}B(A+tI_n)^{-1},\quad A,B\in\mathbf{H}_n^{++},\] 
and use Lemma 3.4 in \cite{sutter2017multivariate} that 
\[\int_0^{+\infty}dt(A^{-1}+tI_n)^{-1}B(A^{-1}+tI_n)^{-1} = \int_{-\infty}^{+\infty}dt\beta_0(t)A^\frac{1+it}{2}BA^\frac{1-it}{2}, \quad A,B\in\mathbf{H}_n^{++},\]
we then further obtain 
\[\phi\big(\exp(A^{(1)}+A^{(2)}+A^{(3)})\big) \leq \phi\big(\exp(A^{(1)})\mathcal{T}_{\exp(-A^{(2)})}[\exp(A^{(3)})]\big).\]
This can be seen as a $k$-trace generalization of Lieb's \cite{LIEB1973267} three-matrix extension of the GT inequality that $\trace[\exp(A+B+C)]\leq \trace[\exp(A)\mathcal{T}_{\exp(-B)}[\exp(C)]].$

\subsection{Monotonicity Preserving and Concavity Preserving}
As mentioned in the review part of \Cref{sec:Main}, Loewner's theorem says that a real-valued function $f$ on $(0,+\infty)$ is operator monotone if and only if it admits an analytic continuation to a Herglotz function. Therefore, the extension of a monotone scalar function to Hermitian matrices is not necessarily operator monotone. For instance, let $f(x)=x^3$, and 
\[A=\left(\begin{array}{cc} 1&0\\0&0\end{array}\right),\quad B=\left(\begin{array}{cc} 2&-1\\-1&1\end{array}\right),\]
then $f$ is monotone increasing, and $A\preceq B$. But neither $A^3\preceq B^3$ nor $A^3\succeq B^3$ is true. However, a composition with trace will preserve the monotonicity. That is, $\trace[f(A)]$ is monotone increasing (or decreasing) in $A$ with respect to Loewner partial order, if $f$ is monotone increasing (or decreasing). Likewise, if $f$ is concave (or convex), its extension to Hermitian matrices is not necessarily operator concave (or convex), but $A\mapsto\trace[f(A)]$ is still concave (or convex). One can see Theorem 2.10 in \cite{carlen2010trace}. This means that, some partial information like trace may preserve monotonicity and concavity. In fact, we will show that for any integer $k$ the partial information $\phi(\cdot)=\trace_k[\cdot]^\frac{1}{k}$ also preserves monotonicity and concavity of scalar functions. But we need to restrict to $f$ that only takes values in $[0,+\infty)$. We need the following lemma for proving concavity preserving.

\begin{lemma}\label{lem:DiagonalMajoring}
For any $A\in \mathbf{H}_n^+$, let $\diag(A)$ denote the diagonal part of $A$, then 
\[\phi(A)\leq\phi(\diag(A)) .\]
\end{lemma}
\begin{proof}
Let $D$ be a $n\times n$ diagonal matrix, whose diagonal entries follow independent Rademacher distributions, i.e.
\[D_{ii} = \left\{\begin{array}{rl}
 1, & \text{with prob. } 0.5, \\
 -1, & \text{with prob. } 0.5.
 \end{array}\right.\]
Since $D^2\equiv I_n$, we have $\phi(A)=\phi(DAD)$.  Also notice that $\mathbb{E}[DAD] = \diag(A)$, since $\mathbb{E}[D_{ii}D_{jj}] = \delta_{ij}$. Then by concavity of $\phi$, we have
\[\phi(A) = \mathbb{E}\phi(DAD)\leq \phi(\mathbb{E}[DAD])=\phi(\diag(A)).\]
\end{proof}

\Cref{lem:DiagonalMajoring} can be proved, instead, using the concept of Majorization. Let $\bm{a}=\{a_i\}_i^n$ and $\bm{b}=\{b_i\}_i^n$ be two sequences, both in descending order, i.e. $a_1\geq a_2\geq \cdots\geq a_n$, $b_1\geq b_2\geq \cdots\geq b_n$. We say $\bm{b}$ majorizes $\bm{a}$, denoted by $\bm{b}\succeq \bm{a}$, if  
\[\sum_{i=1}^kb_i\geq \sum_{i=1}^ka_i,\quad 1\leq k\leq n,\]
and the equality holds for $k=n$. It is not hard to show that, if $\bm{b}\succeq \bm{a}$, and $\bm{a},\bm{b}\subset [0,+\infty)$, then $\prod_{i=1}^db_i\leq \prod_{i=1}^da_i$. Now for any $A\in \mathbf{H}_n^+$, let $\bm{\lambda}=\{\lambda_i\}_{i=1}^n$ and $\bm{a}=\{a_i\}_i^n$ be the eigenvalues and the diagonal entries of $A$ respectively, both in descending order. Then since 
\[\sum_{i=1}^k\lambda_i=\max_{\begin{subarray}{c}\{v_i\}_{i=1}^k\subset \mathbb{C}^n\\ v^*_iv_j = \delta_{ij}\end{subarray}} \sum_{i=1}^kv^*_iAv_i \geq \sum_{i=1}^ke^*_iAe_i = \sum_{i=1}^ka_i,\quad 1\leq k<n\]
and $\sum_{i=1}^n\lambda_i =\trace[A] = \sum_{i=1}^na_i$, we have $\bm{\lambda}\succeq\bm{a}$. Therefore we know that 
\[\det[A]=\prod_{i=1}^n\lambda_i\leq \prod_{i=1}^na_i=\det[\diag(A)].\]
Then using the equivalent definition \eqref{def:k-trace2} of $k$-trace, we have that, for any $1\leq k\leq n$,
\begin{align*}
\trace_k[A] =&\ \sum_{1\leq i_1<i_2<\cdots<i_k\leq n}\det[A_{(i_1\cdots i_k,i_1\cdots i_k)}]\\
\leq&\ \sum_{1\leq i_1<i_2<\cdots<i_k\leq n}A_{i_1i_1}A_{i_2i_2}\cdots A_{i_ki_k},\\
=&\ \trace_k[\diag(A)].
\end{align*}

\begin{thm}\label{thm:Preserving}
Given a function $f:[0,+\infty)\rightarrow [0,+\infty)$, if $f$ is monotone increasing (or decreasing) as a scalar function, then $\phi(f(\cdot))$ is monotone increasing (or decreasing) on $\mathbf{H}_n^+$, in the sense that $\phi(f(A))\geq \phi(f(B))$(or $\phi(f(A))\leq \phi(f(B))$) if $A,B\in \mathbf{H}_n^+,A\succeq B$; if $f$ is concave as a scalar function, then $\phi(f(\cdot))$ is concave on $\mathbf{H}_n^+$. 
\end{thm}
\begin{proof}
We first prove the monotonicity preserving property of $\phi$. For any matrix $A\in \mathbf{H}_n$, we denote by $\lambda_i(A)$ the $i_{\text{th}}$ largest eigenvalue of $A$. For any $A,B\in\mathbf{H}_n^+$, if $A\succeq B$, then $\lambda_i(A)\geq \lambda_i(B), 1\leq i\leq n$. Therefore if $f$ is monotone increasing, we immediately have 
\[\lambda_i(f(A)) = f(\lambda_i(A)) \geq f(\lambda_i(B)) = \lambda_i(f(B)),\quad 1\leq i\leq n,\]
and thus $\phi(f(A))\geq \phi(f(B))$ by definition $\phi$. Similarly, if $f$ is  monotone decreasing, we have
\[\lambda_i(f(A)) = f(\lambda_{n-i+1}(A)) \leq f(\lambda_{n-i+1}(B)) = \lambda_i(f(B)),\quad 1\leq i\leq n,\]
and thus $\phi(f(A))\leq \phi(f(B))$. 

Next we prove the concavity preserving property of $\phi$. Given any $A,B\in \mathbf{H}_n^+$, and any $\tau\in [0,1]$, we define $C=\tau A+(1-\tau)B$. Let $U=[u_1,u_2,\cdots,u_n]$ be a unitary matrix such that the columns are all eigenvectors of $C$, then 
\[U^*f(C)U = f(U^*CU) = f(\diag(U^*CU)).\] 
If $f$ is concave, then 
\[f(u_i^*Cu_i) = f(\tau u_i^*Au_i +(1-\tau)u_i^*Bu_i)\geq \tau f(u_i^*Au_i) + (1-\tau)f(u_i^*Bu_i),\quad 1\leq i\leq n,\]
and thus $f(\diag(U^*CU))\succeq \tau f(\diag(U^*AU)) + (1-\tau)f(\diag(U^*BU))$. Further, for any unit vector $u\in \mathbb{C}^n$, we have
\begin{align*}
f(u^*Au) =&\ f\Big(\sum_{i=1}^n\lambda_i(A)u^*v_iv_i^*u\Big) = f\Big(\sum_{i=1}^n\lambda_j(A)|v_i^*u|^2\Big)\\
\geq&\  \sum_{i=1}^n|v_i^*u|^2f\big(\lambda_i(A)\big) = u^*\Big(\sum_{i=1}^nf\big(\lambda_i(A)\big)v_iv_i^*\Big)u = u^*f(A)u,
\end{align*}
where $v_1,v_2,\cdots,v_n$ are all eigenvectors of $A$, and we have used that $\sum_{i=1}^n|v_i^*u|^2 = \|u\|_2^2=1$. Then we have $f(\diag(U^*AU))\succeq \diag(U^*f(A)U)$. Similar, we have $f(\diag(U^*BU))\succeq \diag(U^*f(B)U)$. Finally, we can compute
\begin{align*}
\phi\big(f(C)\big) =&\ \phi\big(U^*f(C)U\big) \\
=&\ \phi\big(f(\diag(U^*CU))\big) \\
\geq&\ \phi\big(\tau f(\diag(U^*AU)) + (1-\tau)f(\diag(U^*BU))\big) \\
\geq&\ \tau \phi\big(f(\diag(U^*AU))\big) + (1-\tau)\phi\big(f(\diag(U^*BU))\big) \\
\geq&\ \tau \phi\big(\diag(f(U^*AU))\big) + (1-\tau)\phi\big(\diag(f(U^*BU))\big)\\
\geq&\ \tau \phi\big(f(U^*AU)\big) + (1-\tau)\phi\big(f(U^*BU)\big)\\
=&\ \tau \phi\big(U^*f(A)U\big) + (1-\tau)\phi\big(U^*f(B)U\big)\\
=&\ \tau \phi\big(f(A)\big) + (1-\tau)\phi\big(f(B)\big).
\end{align*}
We have used \Cref{lem:DiagonalMajoring} for the last inequality above. Therefore $\phi(f(\cdot))$ is concave on $\mathbf{H}_n^+$.
\end{proof}

\section*{Appendix}

\begin{appendix}
\renewcommand{\thesection}{\Alph{section}}

\section{Mixed Discriminant}
\label{Apdix:MixedDiscriminant}
The mixed discriminant $D(A^{(1)},A^{(2)},\cdots,A^{(n)})$ of $n$ matrices $A^{(1)},A^{(2)},\cdots,A^{(n)}\in \mathbb{C}^{n\times n}$ is defined as 
\begin{equation}
D(A^{(1)},A^{(2)},\cdots,A^{(n)}) = \frac{1}{n!} \sum_{\sigma\in S_n}\det\left[\begin{array}{cccc}
 A_{11}^{(\sigma(1))} & A_{12}^{(\sigma(2))} & \cdots & A_{1n}^{(\sigma(n))} \\
 A_{21}^{(\sigma(1))} & A_{22}^{(\sigma(2))} & \cdots & A_{1n}^{(\sigma(n))} \\
 \vdots & \vdots & \ddots & \vdots\\
 A_{n1}^{(\sigma(1))} & A_{n2}^{(\sigma(2))} & \cdots & A_{nn}^{(\sigma(n))}
\end{array}\right],
\label{eqt:mixdiscriminant}
\end{equation}
where $S_n$ denotes the symmetric group of degree $n$. We here list some basic facts about mixed discriminants. For more properties of mixed discriminants, one may refer to \cite{bapat1989mixed,panov1987some}.
\begin{itemize} 
\item Symmetry: $D(A^{(1)},A^{(2)},\cdots,A^{(n)})$ is symmetric in $A^{(1)},A^{(2)},\cdots,A^{(n)}$, i.e. 
\[D(A^{(1)},A^{(2)},\cdots,A^{(n)}) = D(A^{\sigma(1)},A^{\sigma(2)},\cdots,A^{\sigma(n)}),\quad \sigma\in S_n.\]
\item Multilinearity: for any $\alpha,\beta\in \mathbb{R}$, 
\[D(\alpha A+\beta B,A^{(2)},\cdots,A^{(n)}) = \alpha D(A,A^{(2)},\cdots,A^{(n)}) + \beta D(B,A^{(2)},\cdots,A^{(n)}).\]
\item Positiveness \cite{bapat1989mixed}: If $A^{(1)},A^{(2)},\cdots,A^{(n)}\in \mathbf{H}_n^+$, then $D(A^{(1)},A^{(2)},\cdots,A^{(n)})\geq0$; if $A^{(1)},A^{(2)},\cdots,A^{(n)}\in \mathbf{H}_n^{++}$, then $D(A^{(1)},A^{(2)},\cdots,A^{(n)})>0$.
\end{itemize}

The relation between the mixed discriminant and $\trace_k$ is obvious. If we calculate the mixed discriminant for $k$ copies of $A\in\mathbb{C}^{n\times n}$ and $n-k$ copies of $I_n$, we can find that 
\begin{align}
D(\underbrace{A,\cdots,A}_{k},\underbrace{I_n,\cdots,I_n}_{n-k}) = \binom{n}{k}^{-1}\sum_{1\leq i_1<i_2<\cdots<i_k\leq n}\det[A_{(i_1\cdots i_k,i_1\cdots i_k)}] = \binom{n}{k}^{-1}\trace_k[A].
\end{align}
This is why the mixed discriminant plays an important role in the proof of our main theorem. In particular, we will need the following inequality on mixed discriminant by Alexandrov \cite{aleksandrov1938theory}.

\begin{thm}
\label{thm:AFinequality}
(\textbf{Alexandrov--Fenchel Inequality for Mixed Discriminants})
For any $B\in \mathbf{H}_n$ and any $A,\underbrace{A^{(3)},\cdots,A^{(n)}}_{n-2}\in \mathbf{H}_n^{++}$, we have
\begin{equation}
D(A,B,A^{(3)},\cdots,A^{(n)})^2\geq D(A,A,A^{(3)},\cdots,A^{(n)})D(B,B,A^{(3)},\cdots,A^{(n)}),
\label{eqt:AFinequality}
\end{equation}
with equality if and only if $B=\lambda A$ for some $\lambda\in \mathbb{R}$.
\end{thm}
This theorem originally applied to real symmetric matrices when established in 1937. A proof of its extension to Hermitian matrices can be found in \cite{li2017alexandrov}. By continuity, inequality \eqref{eqt:AFinequality} can extend to the case that $A,A^{(3)},\cdots,A^{(n)}\in \mathbf{H}_n^+$, but the necessity of the condition for equality is no longer valid.

Repeatedly applying the Alexandrov--Fenchel inequality \eqref{eqt:AFinequality} grants us the following corollary.
\begin{corollary}
\label{cor:generalAFinequality}
For any $0\leq l\leq k\leq n$, and any $A,B,\underbrace{A^{(k+1)},\cdots,A^{(n)}}_{n-k}\in \mathbf{H}_n^+$, we have
\begin{align}
&\ D(\underbrace{A,\cdots,A}_{l},\underbrace{B,\cdots,B}_{k-l},\underbrace{A^{(k+1)},\cdots,A^{(n)}}_{n-k})^k\\
\geq&\ D(\underbrace{A,\cdots,A}_{k},\underbrace{A^{(k+1)},\cdots,A^{(n)}}_{n-k})^l\cdot D(\underbrace{B,\cdots,B}_{k},\underbrace{A^{(k+1)},\cdots,A^{(n)}}_{n-k})^{k-l}. \nonumber
\end{align}
\end{corollary}

A direct result of \Cref{cor:generalAFinequality} is the following general Brunn--Minkowski theorem for mixed discriminants.

\begin{corollary}
\label{cor:BMtheorem}
(\textbf{General Brunn--Minkowski Theorem for Mixed Discriminants})
For any $1\leq k\leq n$, and any fixed $\underbrace{A^{(k+1)},\cdots,A^{(n)}}_{n-k}\in \mathbf{H}_n^+$, the function
\begin{align}
\mathbf{H}_n^+\ &\longrightarrow \ \mathbb{R} \nonumber \\
A \ &\longmapsto\ D(\underbrace{A,\cdots,A}_{k},\underbrace{A^{(k+1)},\cdots,A^{(n)}}_{n-k})^\frac{1}{k}
\end{align}
is concave.
\end{corollary}

If we choose $A^{(k+1)},\cdots,A^{(n)}$ to be $n-k$ copies of $I_n$, we can immediately conclude from \Cref{cor:BMtheorem} that the function $A\mapsto \big(\trace_k\big[A\big]\big)^\frac{1}{k}$ is concave on $\mathbf{H}_n^+$. 

\section{Exterior Algebra}
\label{Apdix:ExteriorAlgebra}
We give a brief review of exterior algebras on the vector space $\mathbb{C}^n$. For more details, one may refer to \cite{bishop1980tensor,rotman2010advanced}. For the convenience of our use, the notations in our paper might be different from those in other materials. For any $1\leq k\leq n$, let $\wedge^k(\mathbb{C}^n)$ denote the vector space of the $k_{th}$ exterior algebra of $\mathbb{C}^n$, equipped with the inner product 
\begin{align*}
\langle \cdot,\cdot\rangle_{\wedge^k}:\quad \wedge^k(\mathbb{C}^n)\times\wedge^k(\mathbb{C}^n)\ &\longrightarrow\ \mathbb{C}\\
\langle u_1\wedge\cdots\wedge u_k,v_1\wedge\cdots\wedge v_k\rangle_{\wedge^k}\ &= 
\det\left[\begin{array}{cccc}
 \langle u_1,v_1\rangle & \langle u_1,v_2\rangle  & \cdots & \langle u_1,v_k\rangle \\
 \langle u_2,v_1\rangle & \langle u_2,v_2\rangle & \cdots & \langle u_2,v_k\rangle\\
 \vdots & \vdots & \ddots & \vdots \\
 \langle u_k,v_1\rangle & \langle u_k,v_2\rangle & \cdots & \langle u_k,v_k\rangle
\end{array}\right],
\end{align*}
where $\langle u,v\rangle=u^*v$ is the standard $l_2$ inner product on $\mathbb{C}^n$.

Let $\mathcal{L}(\wedge^k(\mathbb{C}^n))$ denote the space of all linear operators from $\wedge^k(\mathbb{C}^n)$ to itself. For any matrices $A^{(1)},A^{(2)},\cdots,A^{(k)}\in \mathbb{C}^{n\times n}$, we can define an element in $\mathcal{L}(\wedge^k(\mathbb{C}^n))$:
\begin{align}
\label{def:ExteriorOps}
\mathcal{M}^{(k)}(A^{(1)},A^{(2)},\cdots,A^{(k)}):\quad \wedge^k(\mathbb{C}^n) &\ \longrightarrow\ \wedge^k(\mathbb{C}^n)\nonumber \\
v_1\wedge v_2\wedge\cdots\wedge v_k &\ \longmapsto\  \sum_{ \sigma\in S_k} A^{(\sigma(1))}v_1\wedge A^{(\sigma(2))}v_2\wedge\cdots\wedge A^{(\sigma(k))}v_k,
\end{align}
where $S_k$ is the symmetric group of degree $k$. Apparently, the map $\mathcal{M}^{(k)}(A^{(1)},A^{(2)},\cdots,A^{(k)})$ is symmetric in $A^{(1)},A^{(2)},\cdots,A^{(k)}$, and linear in each single $A^{(i)}$. For simplicity, we will use the following notations for any matrix $A,B\in \mathbb{C}^{n\times n}$:
\begin{subequations}
\begin{align}
\MM_0^{(k)}(A) &= \frac{1}{k!}\mathcal{M}^{(k)}(A,\cdots,A),\\
\MM_1^{(k)}(A;B) &= \frac{1}{(k-1)!}\mathcal{M}^{(k)}(A,B,\cdots,B).
\end{align}
\end{subequations}
Obviously the identity operator in $\mathcal{L}(\wedge^k(\mathbb{C}^n))$ is $\MM_0(I_n)$. Note that due to the skew-symmetric property of exterior algebra, the spectrum of $\MM_0^{(k)}(A)$ is $\{\lambda_{i_1}\lambda_{i_2}\cdots\lambda_{i_k}\}_{1\leq i_1<i_2<\cdots<i_k\leq n}$, where $\lambda_1,\lambda_2,\cdots,\lambda_n$ are the eigenvalues of $A$. It is easy to verify that $\MM_0^{(k)}$ has the following properties:
\begin{itemize}
\item Invertibility: If $A\in \mathbb{C}^{n\times n}$ is invertible, then $(\MM_0^{(k)}(A))^{-1} = \MM_0^{(k)}(A^{-1})$.
\item Adjoint: For any $A\in \mathbb{C}^{n\times n}$, $(\MM_0^{(k)}(A))^*=\MM_0^{(k)}(A^*)$, with respect to the inner product $\langle \cdot,\cdot\rangle_{\wedge^k}$. 
\item Power: For any $A\in\mathbf{H}_n$ and any $t\in\mathbb{C}$, $\big(\MM_0^{(k)}(A)\big)^t = \MM_0^{(k)}(A^t)$.
\item Positiveness: If $A\in \mathbf{H}_n$, then $\MM_0^{(k)}(A)$ is Hermitian; if $A\in \mathbf{H}_n^+$, then $\MM_0^{(k)}(A)\succeq \mathbf{0}$; if $A\in \mathbf{H}_n^{++}$, then $\MM_0^{(k)}(A)\succ\mathbf{0}$. 

\item Product: For any $A,B\in \mathbb{C}^{n\times n}$, $\MM_0^{(k)}(AB) = \MM_0^{(k)}(A)\MM_0^{(k)}(B)$.
\end{itemize}
Using these properties, one can check that 
\[|\MM_0^{(k)}(A)|=\big((\MM_0^{(k)}(A))^*\MM_0^{(k)}(A)\big)^\frac{1}{2}=\MM_0^{(k)}\big((A^*A)^\frac{1}{2}\big)=\MM_0^{(k)}(|A|).\]

Next we consider the natural basis of $\wedge^k(\mathbb{C}^n)$, 
\[\{e_{i_1}\wedge e_{i_2}\wedge\cdots\wedge e_{i_k} \}_{1\leq i_1<i_2<\cdots<i_k\leq n},\]
which is orthogonal under the inner product $\langle \cdot,\cdot\rangle_{\wedge^k}$. 
Then the trace function on $\mathcal{L}(\wedge^k(\mathbb{C}^n))$ is defined as 
\begin{align}
\trace:\quad \mathcal{L}(\wedge^k(\mathbb{C}^n))\ &\longrightarrow\ \mathbb{C}\nonumber\\
\trace\big[\mathcal{F} \big]\ &= \ \sum_{1\leq i_1<i_2<\cdots<i_k\leq n} \langle e_{i_1}\wedge e_{i_2}\wedge\cdots\wedge e_{i_k},\mathcal{F}(e_{i_1}\wedge e_{i_2}\wedge\cdots\wedge e_{i_k}) \rangle_{\wedge^k}.
\end{align}
It is not hard to check that this trace function is also invariant under cyclic permutation, i.e. $\trace\big[\mathcal{F}\mathcal{G}\big]=\trace\big[\mathcal{G}\mathcal{F}\big]$ for any $\mathcal{F},\mathcal{G}\in \mathcal{L}(\wedge^k(\mathbb{C}^n))$. Then for any $A^{(1)},\cdots,A^{(k)}\in \mathbb{C}^{n\times n}$, the trace $\trace[\mathcal{M}^{(k)}(A^{(1)},\cdots,A^{(k)})]$ coincides with the definition of the mixed discriminant, as one can check that 
\begin{align}
\trace\big[\MM^{(k)}(A^{(1)},\cdots,A^{(k)})\big]=&\ \sum_{\sigma\in S_k}\sum_{1\leq i_1<\cdots<i_k\leq n} \langle e_{i_1}\wedge\cdots\wedge e_{i_k},A^{(\sigma(1))}e_{i_1}\wedge\cdots\wedge A^{(\sigma(k))}e_{i_k} \rangle_{\wedge^k}\nonumber\\
=&\ \frac{n!}{(n-k)!}D(A^{(1)},\cdots,A^{(k)},\underbrace{I_n,\cdots,I_n}_{n-k}).
\label{eqt:extrace2mixdiscriminant}
\end{align}
From this observation, we can now express the $k$-trace of a matrix $A\in \mathbb{C}^{n\times n}$ as 
\begin{equation}
\trace_k[A] = \trace\big[\MM_0^{(k)}(A)\big].
\label{eqt:extrace2ktrace}
\end{equation} 
For those who are familiar with exterior algebra, it is clear that the spectrum of $\MM_0^{(k)}$ is just $\{\lambda_{i_1}\lambda_{i_2}\cdots\lambda_{i_k}\}_{1\leq i_1<i_2<\cdots<i_k\leq n}$, where $\lambda_1,\lambda_2,\cdots,\lambda_n$ are the eigenvalues of $A$. So in this way it is more convenient to see that $\trace\big[\MM_0^{(k)}(A)\big] = \mathrm{sum}(\text{spectrum of }\MM_0^{(k)}(A)) = \sum_{1\leq i_1<\cdots<i_k\leq n}\lambda_{i_1}\lambda_{i_2}\cdots\lambda_{i_k} =\trace_k[A]$. 

We here provide an alternative of \Cref{lem:ktraceGT} using the following lemma.

\begin{lemma}
\label{lem:lemmaB1}
For any $A\in\mathbf{H}_n$, we have
\[\MM_0^{(k)}\big(\exp(A)\big) = \exp\big(\MM_1^{(k)}(A;I_n)\big).\]
\end{lemma} 
\begin{proof} We need to show that for any $v_1\wedge v_2\wedge\cdots\wedge v_k\in\wedge^k(\mathbb{C}^n)$, 
\begin{equation}
\label{eqt:lemmaB1}
\MM_0^{(k)}\big(\exp(A)\big)(v_1\wedge v_2\wedge\cdots\wedge v_k) = \exp\big(\MM_1^{(k)}(A;I_n)\big)(v_1\wedge v_2\wedge\cdots\wedge v_k).
\end{equation}
We use Taylor expansion of $e^x$ to expand 
\[\MM_0^{(k)}\big(\exp(A)\big) = \MM_0^{(k)}\Big(\sum_{j=0}^{+\infty}\frac{1}{j!}A^j\Big),\quad \exp\big(\MM_1^{(k)}(A;I_n)\big) = \sum_{j=0}^{+\infty}\frac{1}{j!}\big(\MM_1^{(k)}(A;I_n)\big)^j.\]
Then for any integers $j_1,j_2,\dots,j_k\geq0$, the coefficient of the term $A^{j_1}v_1\wedge A^{j_2}v_2\wedge\cdots\wedge A^{j_k}v_k$ in the left hand side of \eqref{eqt:lemmaB1} is 
\[\frac{1}{j_1!j_2!\cdots j_k!},\]
and the coefficient of the same term in the right hand side of \eqref{eqt:lemmaB1} is also 
\[\frac{1}{J!}\binom{J}{j_1}\binom{J-j_1}{j_2}\cdots\binom{J-j_1-j_2-\cdots-j_{k-1}}{j_k} = \frac{1}{j_1!j_2!\cdots j_k!}\quad (J=j_1+j_2+\cdots+j_k).\]
\end{proof}

\begin{proof}[\rm\textbf{An alternative proof of \Cref{lem:ktraceGT}}] Using \Cref{lem:lemmaB1} and the original GT inequality for normal trace, we have 
\begin{align*}
\trace_k[\exp(A+B)]=&\ \trace\big[\MM_0^{(k)}\big(\exp(A+B)\big)\big] \\
=&\ \trace\big[\exp\big(\MM_1^{(k)}(A+B;I_n)\big)\big]\\
=&\ \trace\big[\exp\big(\MM_1^{(k)}(A;I_n)+\MM_1^{(k)}(B;I_n)\big)\big]\\
\leq&\ \trace\big[\exp\big(\MM_1^{(k)}(A;I_n)\big)\exp\big(\MM_1^{(k)}(B;I_n)\big)\big]\\
=&\ \trace\big[\MM_0^{(k)}\big(\exp(A)\big)\MM_0^{(k)}\big(\exp(B)\big)\big]\\
=&\ \trace_k\big[\exp(A)\exp(B)\big],
\end{align*}
where we have used that $\MM_1^{(k)}(X;I_n)$ is linear in $X$. As shown by Petz \cite{petz1994survey}, in the original GT inequality, the equality $\trace[\exp(A+B)]=\trace[\exp(A)\exp(B)]$ holds for $A,B\in\mathbf{H}_n$ if and only if $AB=BA$. Therefore, according to our calculation above, the equality $\trace_k[\exp(A+B)]=\trace_k[\exp(A)\exp(B)]$ holds if and only if 
\begin{equation}\label{eqt:commutibility}
\MM_1^{(k)}(A;I_n)\MM_1^{(k)}(B;I_n) = \MM_1^{(k)}(B;I_n)\MM_1^{(k)}(A;I_n).
\end{equation}
However, one can check by definition that \eqref{eqt:commutibility} is true if and only if $AB=BA$.
\end{proof}

\section{Complex Interpolation}
\label{Apdix:ComplexInterpolation}
The idea of complex interpolation originates from an important result in harmonic analysis, the Hadamard three-lines theorem \cite{hadamard1899theoreme}, that if $f(z)$ is uniformly bounded on $\mathcal{S}=\{z\in\mathbb{C}:0\leq \mathrm{Re}(z)\leq1\}$, holomorphic in the interior and continuous on the boundary, then $g(x)=\log\sup_{y}|f(x+iy)|$ is a convex function on $[0,1]$. Hirschman \cite{hirschman1952convexity} improved this theorem to the following.

\begin{thm}[Hirschman] Let $f(z)$ be uniformly bounded on $\mathcal{S}$, holomorphic in the interior and continuous on the boundary. Then for $\theta\in(0,1)$, we have
\[\log|f(\theta)|\leq \int_{-\infty}^{+\infty}dt\big(\beta_{1-\theta}(t)\log|f(it)|^{1-\theta}+\beta_\theta(t)\log|f(1+it)|^\theta\big).\]
Moreover, the assumption that $f(z)$ is uniformly bounded can be relaxed to 
\[\log|f(z)|\leq Ce^{a\mathrm{Im}(z)},\quad \forall z\in \mathcal{S},\quad \text{for some constants}\ C<+\infty,\ a<\pi. \]
\end{thm}

Stein \cite{stein1956interpolation} further generalized this complex interpolation theory to interpolation of linear operators as described in \Cref{thm:SHInterpolation}. Our key lemma, \Cref{lem:KeyLemma}, is then a natural extension of \Cref{thm:SHInterpolation} from Schatten norm to $k$-trace. We provide a proof of \Cref{lem:KeyLemma} as follows.

\begin{proof}[\rm\textbf{Proof of \Cref{lem:KeyLemma}}]
For any $X\in\mathbb{C}^{n\times n}$ and $p\in[1,+\infty)$, we have that
\[\MM_0^{(k)}\big(|X|^p\big)=\MM_0^{(k)}\big((X^*X)^\frac{p}{2}\big) = \big(\MM_0^{(k)}(X^*X)\big)^\frac{p}{2}=\big((\MM_0^{(k)}(X))^*\MM_0^{(k)}(X)\big)^\frac{p}{2}=\big|\MM_0^{(k)}(X)\big|^p,\]
and thus 
\[\trace_k\big[|X|^p\big]^\frac{1}{p} = \trace\big[\MM_0^{(k)}\big(|X|^p\big)\big]^\frac{1}{p} = \trace\big[\big|\MM_0^{(k)}(X)\big|^p\big]^\frac{1}{p} = \big\|\MM_0^{(k)}(X)\big\|_p.\]
The above equality also holds for $p\rightarrow +\infty$ since we are dealing with finite dimensional operators. If $G(z)$ is holomorphic in the interior of $\mathcal{S}$ and continuous on the boundary, then so is $\MM_0^{(k)}\big(G(z)\big)$. And if $\|G(z)\|$ is uniformly bounded on $\mathcal{S}$, then $\|\MM_0^{(k)}\big(G(z)\big)\|_{p_{\mathrm{Re}(z)}}$ is also uniformly bounded on $\mathcal{S}$, since all norms are equivalent for finite dimensional operators. Therefore we can use \Cref{thm:SHInterpolation} with $G(z)$ replaced by $\MM_0^{(k)}\big(G(z)\big)$ to get 
\begin{align*}
&\ \log\big(\trace_k\big[|G(\theta)|^{p_\theta}\big]^\frac{1}{p_\theta}\big) \\
\leq&\  \int_{-\infty}^{+\infty}dt\Big(\beta_{1-\theta}(t)\log\big(\trace_k\big[|G(it)|^{p_0}\big]^\frac{1-\theta}{p_0}\big)+\beta_\theta(t)\log\big(\trace_k\big[|G(1+it)|^{p_1}\big]^\frac{\theta}{p_1}\big)\Big).
\end{align*}
We then multiply both sides by $\frac{1}{k}$ to obtain \eqref{eqt:KeyIneq}. 
\end{proof}

We remark that the operator interpolation inequality in \Cref{thm:SHInterpolation} is interestingly powerful and user-friendly for proving matrix inequalities, as we have seen in the proofs of \Cref{lem:GeneralEpstein} and \Cref{thm:GeneralLiebConcavity}. In fact this technique can also prove many fundamental results in matrix theories. We here show one more example to see how we can use the trick of interpolation to prove that the map $A\mapsto A^r$ is operator concave on $\mathbf{H}_n^+$ for $r\in(0,1]$. Note that, in general, this result is proved by using an integral expression for $A^r$ (e.g. see \cite{carlen2010trace}).

We need to show that for any $A,B\in\mathbf{H}_n^+$, $\tau\in[0,1]$,
\begin{equation}\label{eqt:A^rConcave}
\tau A^r +(1-\tau) B^r\leq C^r,
\end{equation}
where $C=\tau A+(1-\tau)B$. We may assume that $C$ is invertible. The case when $C$ is not invertible can be handled by continuity. Notice that $X\preceq I_n$ if and only if $\trace[K^*XK]\leq \trace[K^*K]$ for all $K\in\mathbb{C}^{n\times n}$. \eqref{eqt:A^rConcave} is then equivalent to the statement that 
\[\tau \trace[K^*C^{-\frac{r}{2}}A^rC^{-\frac{r}{2}}K] + (1-\tau) \trace[K^*C^{-\frac{r}{2}}B^rC^{-\frac{r}{2}}K]\leq \trace[K^*K],\quad \forall K\in\mathbb{C}^{n\times n}.\] 
Now we fix $K$ and define 
\[G_X(z) = X^\frac{z}{2}C^{-\frac{z}{2}}K,\quad X=A,B,\]
so we have
\[\trace[K^*C^{-\frac{r}{2}}X^rC^{-\frac{r}{2}}K] = \|G_X(r)\|_2^2.\]
$G_X(z)$ is holomorphic in the interior of $\mathcal{S}$ and continuous on the boundary, and $\|G_X(z)\|$ is uniformly bounded on $\mathcal{S}$. We then use inequality \eqref{eqt:SHInterpolation} in \Cref{thm:SHInterpolation} with $\theta = r,p_\theta=p_0=p_1=2$ to obtain
\[\|G_X(r)\|_2^2 \leq  \int_{-\infty}^{+\infty}dt\Big((1-r)\beta_{1-r}(t)\|G_X(it)\|_2^2+r\beta_r(t)\|G_X(1+it)\|_2^{2}\Big).\]
We have again used Jensen's inequality to get rid of the logarithms. For each $t\in\mathbb{R}$, we have that 
\[\|G_X(it)\|_2^2 = \trace[K^*C^\frac{it}{2}X^{-\frac{it}{2}}X^\frac{it}{2}C^{-\frac{it}{2}}K]=\trace[K^*K],\]
and 
\[\|G_X(1+it)\|_2^2 = \trace[K^*C^{-\frac{1-it}{2}}X^\frac{1-it}{2}X^\frac{1+it}{2}C^{-\frac{1+it}{2}}K]=\trace[K^*C^{-\frac{1-it}{2}}XC^{-\frac{1+it}{2}}K].\]
We then have
\[\tau \|G_A(1+it)\|_2^2 + (1-\tau) \|G_B(1+it)\|_2^2 = \trace[K^*C^{-\frac{1-it}{2}}(\tau A+(1-\tau)B)C^{-\frac{1+it}{2}}K] = \trace[K^*K].\]
Finally we obtain
\[\tau \|G_A(r)\|_2^2 + (1-\tau)\|G_B(r)\|_2^2\leq \trace[K^*K].\] 

\section{Homogeneous Convex/Concave Functions}
\label{Apdix:Homogeneous}
\begin{lemma}\label{lem:homogeneous} 
Let $\mathcal{C}$ be a convex cone in some linear space, i.e $\mathcal{C} = conv(\mathcal{C})$ and $C=\lambda\mathcal{C}$ for any $\lambda>0$. Let function $f:\mathcal{C}\rightarrow[0,+\infty)$ be positively homogeneous of order $1$, i.e. $f(\lambda x) = \lambda f(x)$, for any $x\in \mathcal{C}$ and $\lambda> 0$. Then for any $s\in(0,1)$, $f(x)$ is concave if and only if $f(x)^s$ is concave; for any $s\in(1,+\infty)$, $f(x)$ is convex if and only if $f(x)^s$ is convex.
\end{lemma} 

In general this lemma is proved via an argument of level sets. Here we provide a more direct proof.

\begin{proof}
One direction is trivial. If $f(x)$ is concave, then $f(x)^s$ is concave for $s\in(0,1)$, since $(\cdot)^s$ is concave and monotone increasing. Conversely, if $f(x)^s$ is concave for some $s\in(0,1)$, then $f(\tau x+(1-\tau)y)^s\geq \tau f(x)^s+(1-\tau)f(y)^s$, for any $x,y\in \mathcal{C},\tau\in[0,1]$. Now given any fixed $x,y\in \mathcal{C},\tau\in[0,1]$, we need to show that $\tau f(x)+(1-\tau)f(y)\leq f(\tau x+(1-\tau)y)$. If $f(x)=f(y)=0$, then we are done. Otherwise, we may assume that $f(x)>0$, and define $M = \tau f(x)+(1-\tau)(f(y) +\epsilon)$ for some $\epsilon>0$(this $\epsilon$ is not necessary if $f(y)>0$). We then have
\begin{align*}
f(\tau x+(1-\tau)y) =&\ f\left(\frac{\tau f(x)}{M}\frac{Mx}{f(x)}+\frac{(1-\tau)(f(y)+\epsilon)}{M}\frac{My}{f(y)+\epsilon}\right)^{s\cdot\frac{1}{s}}\\
\geq&\  \left(\frac{\tau f(x)}{M}f\left(\frac{Mx}{f(x)}\right)^s + \frac{(1-\tau)(f(y)+\epsilon)}{M}f\left(\frac{My}{f(y)+\epsilon}\right)^s\right)^\frac{1}{s}\\
=&\ \left(\frac{\tau f(x)}{M}M^s + \frac{(1-\tau)(f(y)+\epsilon)^{1-s}f(y)^s}{M}M^s\right)^\frac{1}{s} \\
=&\ M\left(\frac{\tau f(x)+ (1-\tau)(f(y)+\epsilon)^{1-s}f(y)^s}{\tau f(x)+(1-\tau)(f(y) +\epsilon)}\right)^\frac{1}{s}.
\end{align*}
We then take $\epsilon\rightarrow 0$ to obtain $f(\tau x+(1-\tau)y)\geq \tau f(x)+(1-\tau)f(y)$. Therefore $f(x)$ is concave. The convexity part can be proved similarly.
\end{proof}

\end{appendix}

\section*{Acknowledgment}
The research was in part supported by the NSF Grant DMS-1613861. The author would like to thank Thomas Y. Hou for his wholehearted mentoring and supporting.

\section*{References}
\bibliographystyle{elsarticle-num}
\bibliography{reference}

\end{document}